\documentclass[10pt]{article}

\author{Benjamin McKay, University College Cork \\ 
Alexey Pokrovskiy, London School of Economics}
\title{Locally homogeneous structures on Hopf surfaces}
\date{July 11, 2008}

\usepackage[latin1]{inputenc}
\usepackage{array}
\usepackage{amsfonts}
\usepackage{amsmath}
\usepackage{amssymb}
\usepackage{amsthm}
\usepackage{fontenc}
\usepackage{graphicx}
\usepackage{lscape}
\usepackage{booktabs}
\usepackage{ifthen}
\usepackage{varioref}
\usepackage{tikz}

\usepackage[pdftex]{hyperref}

\newtheorem{theorem}{Theorem}[section]

\newtheorem{lemma}[theorem]{Lemma}
\newtheorem{corollary}[theorem]{Corollary}
\newtheorem{proposition}[theorem]{Proposition}
\theoremstyle{remark}
{%
    \newtheorem{example}[theorem]{Example}
}%

\newtheorem{definition}[theorem]{Definition}
\newtheorem{remark}[theorem]{Remark}

\newcommand{\C}[1]{\ensuremath{\mathbb{C}^{#1}}}
\newcommand{\R}[1]{\ensuremath{\mathbb{R}^{#1}}}
\newcommand{\Z}[1]{\ensuremath{\mathbb{Z}^{#1}}}

\newcommand{\pd}[2]{\ensuremath{\frac{\partial #1}{\partial #2}}}
\newcommand{\OO}[1]{
  \ensuremath{
    \mathcal{O}
    \ifthenelse{\equal{#1}{0}}
      {}
      {\left({#1}\right)}
  }
}
\newcommand{\OOp}[2]{
  \ensuremath{
    \mathcal{O}
    \ifthenelse{\equal{#1}{0}}
      {}
      {\left({#1}\right)}
    \ifthenelse{\equal{#2}{1}}
      {}
      {^{\oplus{#2}}}
  }
}

\newcommand{\Proj}[1]{\ensuremath{\mathbb{P}^{#1}}}
\newcommand{\Sym}[2]{\ensuremath{\operatorname{Sym}^{#1}\left(#2\right)}}
\newcommand{\GL}[1]{\ensuremath{\operatorname{GL}\left(#1\right)}}
\newcommand{\SL}[1]{\ensuremath{\operatorname{SL}\left(#1\right)}}
\newcommand{\PSL}[1]{\ensuremath{\mathbb{P}\operatorname{GL}\left(#1\right)}}
\newlength{\setBracketHeight}
\newcommand{\SetSuchThat}[2]{
  \settoheight{\setBracketHeight}{\ensuremath{#2}}
  \ensuremath{\left\{\left.{#1\rule{0cm}{\setBracketHeight}}\,
      \right|\,{#2}\right\}}}
\newcounter{remarkCounter}
\newcommand{\Gn}{\ensuremath{
\left(\GL{2,\C{}}/\text{$n$-th roots of 1}\right) \rtimes \Sym{n}{\C{2}}^*}}
\newcommand{\Hn}{\ensuremath{\SetSuchThat{(g,p) \in G}{g=
\begin{pmatrix}
a & b \\
0 & d
\end{pmatrix},
\ p(1,0)=0
} }}
\newcommand{\dev}{\operatorname{dev}}
\newcommand{\hol}{\operatorname{hol}}
\newcommand{\yes}{\checkmark}
\newcommand{\no}{{\textsf{x}}}

\begin{document}
\maketitle
 \begin{abstract}
We study holomorphic locally homogeneous geometric structures modelled on line bundles over
the projective line. We classify these structures on primary Hopf surfaces. We write out the 
developing map and holonomy morphism of each of these structures explicitly on each primary Hopf surface.
\end{abstract}
\tableofcontents

\section{Introduction}

\subsection{The problem}

\begin{definition}
Suppose that $M$ is a manifold and that $G/H$ is a homogeneous space.
A \emph{$G/H$-structure} on $M$ is a maximal choice of
coordinates on $M$ valued in $G/H$, with transition maps given by action of elements of $G$.
A $G/H$-structure is also called a locally homogeneous structure modelled on $G/H$.
\end{definition}
There is a great deal known about $G/H$-structures on compact complex surfaces, as long as
$H$ is compact; see Wall \cite{Wall:1985,Wall:1986}.
We will suppose instead that
$G/H$ is a complex homogeneous space (i.e. that $G$ is a complex Lie group
and $H \subset G$ is a closed complex Lie subgroup), and that the $G/H$-structure
is holomorphic (i.e. the coordinates valued in $G/H$ are all holomorphic maps). A 
homogeneous space $G/H$ is \emph{primitive} if $G$ does not preserve
a foliation on $G/H$. A locally homogeneous structure is called
\emph{primitive} if its model is. The primitive
holomorphic locally homogeneous structures on
compact complex surfaces are classified; see Klingler \cite{Klingler:1998}. 
The imprimitive are a mystery, although the foliations are roughly
classified; see Brunella \cite{Brunella:1997}.
This paper will classify explicitly a particular family of imprimitive 
holomorphic locally homogeneous structures
(the $\OO{n}$-structures) on a particular family of compact complex
surfaces (the primary Hopf surfaces). The technique consists largely of elementary power
series calculations using Weierstrass polynomials in different
coordinate charts. Along the way we develop a systematic
machinery for computations on Hopf surfaces. 
Our aim in this paper is to develop the tools needed to eventually classify all holomorphic locally homogeneous
structures on all compact complex surfaces.
Although the arguments of this paper are disappointingly
complicated, the results are simple and surprising.
The authors believe that uncovering these results is an essential
step in the large and important programme of understanding
geometry of locally homogeneous structures 
on low dimensional manifolds.

\subsection{Organization of this paper}

Before we can study geometric structures
on Hopf surfaces, we will need to
review the known results on 
cohomology of line bundles on Hopf surfaces, 
and also classify
the flat $\Proj{1}$-bundles
on Hopf surfaces. We complete
this in sections~\ref{sec:Survey}
and \ref{sec:VectorBundles}.

In section~\ref{sec:GeometricStructures}, 
we define the concept of locally homogeneous 
geometric structure,
and we explain the simplifications to
the general theory that occur on
Hopf surfaces. 

The geometric structures in this paper
are modelled on the total
space $\OO{n}$ of the usual holomorphic
line bundle $\OO{n} \to \Proj{1}$.
We write $\OO{n}$ as $G/H$
for suitable groups $G$ and $H$. We explain how
$\OO{n}$-structures can be encoded as
ordinary differential equations in 
section~\ref{sec:GeometricStructures}.
The group $G$ acting on
$\OO{n}$ is complicated,
and we need to unravel its conjugacy
classes in some detail 
in section~\ref{section:TheModel}.

In section~\ref{section:Examples},
we write out explicit expressions in coordinates for each of the 
$\OO{n}$-structures on each Hopf surface. Unfortunately for our study,
certain Hopf surfaces (known as hyperresonant Hopf surfaces) have large and
complicated families of $\OO{n}$-structures,
depending on arbitrarily large families of parameters,
which are responsible for the length
and complexity of this paper. The hyperresonant structures are the
only surprise in this paper, having no apparent geometric description.
Section~\ref{section:Classification} proves that the various $\OO{n}$-structures 
that we have explicitly written out  are the only ones that any Hopf surface can bear.
Section~\ref{section:Inducement} presents some preliminary results
on locally homogeneous geometric structures inducing these $\OO{n}$-structures.

This material is based upon works supported by the Science Foundation Ireland under Grant No. MATF634.

\section{Definition and survey of Hopf surfaces}\label{sec:Survey}

\subsection{The Poincar\'e domain}\label{subsection:PoincareDomain}

Suppose that $F : \C{2} \to \C{2}$ is a biholomorphism fixing the origin. 
Suppose moreover that all eigenvalues $\lambda$ of $F'(0)$ satisfy 
$|\lambda|<1$; $F$ is said to lie in the \emph{Poincar\'e domain}. By the 
Poincar\'e--Dulac theorem (see \cite{Arnold:1988} p. 192), 
$F$ is conjugate by a biholomorphism of $\C{2}$ to a map of precisely one of the two forms
\begin{alignat*}{4}
 F(z) &= \left(\lambda_1 z_1, \lambda_2 z_2\right), 
\ & 0 < \left|\lambda_2\right| 
\le \left|\lambda_1\right| < 1 
& \text{ or } \\
 F(z) &= \left(\lambda z_1, \lambda^m z_2 + z_1^m\right), 
\ & 0 < \left|\lambda\right| < 1, m \ge 1. &
\end{alignat*}
In the first case, $F$ is called \emph{diagonal}. 
In the second case, $F$ is called \emph{exceptional}
and $m$ is an integer which we 
will call the \emph{degree} of $F$. 
We will not quite follow \cite{Mall:1998} 
in describing a diagonal map as:
\begin{center}
\begin{tabular}{ll}
homothetic & $\lambda_1=\lambda_2$ \\
hyperresonant & $\lambda_1^{m_1} = \lambda_2^{m_2}$ , 
some $m_1 \ge m_2 \ge 1$ integers \\
generic & otherwise.
\end{tabular}
\end{center}

In this article (breaking from tradition) we will 
consider homotheties to be special cases of 
hyperresonant maps, rather than requiring 
that a hyperresonant map have $m_1, m_2 \ge 2$. 
If a map $F$ is hyperresonant, then we will 
refer to the pair $m_1,m_2$ of integers for 
which $\lambda_1^{m_1} = \lambda_2^{m_2}$ and 
for which $m_1$ (and hence $m_2$) has the 
smallest possible positive value, as the 
\emph{hyperresonance} of $F$.

\subsection{Hopf surfaces}

A \emph{Hopf surface} is a compact complex 
surface covered by $\C{2} \setminus 0$. For 
example, for any $F$ in the Poincar\'e domain, 
let $S_F$ be the quotient 
$\left(\C{2} \setminus 0\right)/(z \sim F(z))$. 
Hopf surfaces of the form $S_F$ are called 
\emph{primary}. Every Hopf surface admits a 
finite covering by a primary Hopf surface 
(see \cite{Kodaira:1966} p. 696). From now on, 
when we refer to a Hopf surface, we will always 
assume that it is primary. Two Hopf surfaces are 
biholomorphic just when the associated 
biholomorphisms of $\C{2}$ are conjugate by a 
biholomorphism. Any term used to describe the 
map $F$ will also be used to describe $S_F$; 
for example a Hopf surface is called 
\emph{linear} or \emph{diagonal} or 
\emph{resonant}, etc. if the map $F$ is. 

An example: if $F(z)=\frac{1}{2}z$, then 
clearly $S_F$ is diffeomorphic to 
$S^3 \times S^1$. Every Hopf surface can be 
smoothly deformed into this one, through a 
family of Hopf surfaces, so all Hopf 
surfaces are diffeomorphic to $S^3 \times S^1$. 

\subsection{Biholomorphism groups of Hopf surfaces}
The biholomorphism groups of Hopf 
surfaces are well known \cite{Namba:1974,Wehler:1982}:
\begin{center}
\begin{tabular}{ll}
$F\left(z_1,z_2\right)$ & $\operatorname{Bihol} S_F$ \\ \midrule
homothety & invertible linear maps \\
nonhomothetic diagonal linear & invertible diagonal linear maps \\
$\left(\lambda z_1, \lambda^m z_2 + z_1^m \right)$ &
$\left(z_1,z_2\right) \mapsto \left(a z_1, a^m z_2 + b z_1^m \right)$
\end{tabular}
\end{center}
\par\noindent%
with $a \ne 0$ and $b$ arbitrary complex constants.

\subsection{Meromorphic functions}
\begin{definition}
A Weierstrass polynomial
$W\left(z_1,z_2\right)$ is a polynomial
in $z_1$, with
coefficients holomorphic functions of $z_2$,
so that there is some point of 
the complex line $z_1=0$ at which 
$W\left(z_1,z_2\right) \ne 0$.
\end{definition}

\begin{lemma}\label{lemma:MeromorphicWeierstrass}
Suppose that $U \subset \C{2}$ is
an open neighborhood of the origin, 
and that $f$ is a meromorphic function on
$U \setminus 0$. 
Then there is an open neighborhood $U' \subset \C{2}$
of the origin, with $U' \subset U$, an integer $k$,
relatively prime Weierstrass polynomials
$W_1$ and $W_2$ on $U'$,  
and a function $h$ holomorphic on
$U'$, so that 
\[
f = h \, z_1^k \, \frac{W_1}{W_2}
\]
in $U' \setminus 0$.
The neighborhood $U'$ is not uniquely determined, but once $U'$ is 
chosen then the rest is uniquely determined, i.e. any two such representations
must agree on $U'$.
\end{lemma}
\begin{proof}
By Levi's theorem (see \cite{Taylor:1939}) any meromorphic function on 
$U \setminus 0$ extends uniquely to a meromorphic function $U$. We can then find
a possibly smaller neighborhood $U'$ of $0$ on which $f=h_1/h_2$ is a ratio
of holomorphic functions. Write out Weierstrass polynomial
factorizations of $h_1$ and $h_2$; these exist by the Weierstrass preparation
theorem \cite{Hormander:1990} p. 157.
The functions involved are all uniquely determined by uniqueness of
Weierstrass polynomial factorization of holomorphic functions.
\end{proof}

The meromorphic functions on 
Hopf surfaces are well known \cite{Mall:1998}:
\begin{lemma}\label{lemma:MeromorphicFunctionsOnHopfSurfaces}
\begin{center}
\begin{tabular}{lll}
$F$ & 
$\C{}\left(S_F\right)$ & 
\\ \midrule
hyperresonant &  
$\C{}\left(\frac{z_1^{m_1}}{z_2^{m_2}}\right)$ & 
if $\lambda_1^{m_1} = \lambda_2^{m_2}$ \\
generic & \C{} \\
exceptional & \C{} 
\end{tabular}
\end{center}
\end{lemma}
\begin{remark}
An simpler but incorrect proof of this lemma has been given in the 
literature; \cite{Kodaira:1966} p. 697 and 
\cite{Barth/Peters/VanDenVen:1984} p. 226. For $F$ 
exceptional or generic, the arguments 
work perfectly well, and yield the indicated 
meromorphic sections. For $F$ hyperresonant,
it turns out that we will need a little more work,
as will be clarified below.
These authors each claim that if $f$ is a meromorphic
function on a Hopf surface, then $z_1^N f$ is holomorphic 
on $\C{2}$, for large enough $N$, which is not true for
\[
 f=\frac{z_1+z_2}{z_1-z_2}
\]
even though this function $f$ is meromorphic 
on the Hopf surface $S_{1/2}$.
\end{remark}
\begin{proof}
Suppose that $f$ is a meromorphic function 
on a Hopf surface $S_F$. Treat $f$ as an $F$-invariant
meromorphic function on $\C{2}$. By lemma~\vref{lemma:MeromorphicWeierstrass}, near the origin
$f=\frac{h \, W_1}{W_2}$, with $W_1$ and $W_2$ 
uniquely determined Weierstrass polynomials 
in $z_1$, and $h$ nowhere vanishing and 
holomorphic. Under action of $F$, these 
Weierstrass polynomials get transformed into 
new Weierstrass polynomials in $z_1$, up to 
scaling. By uniqueness of Weierstrass 
polynomials for the numerator and denominator 
of $f$, $F$ must just scale each Weierstrass 
polynomial. But then $h$ must also only get 
rescaled. Hence $h, W_1$ and $W_2$ are 
themselves sections of various line bundles 
on the Hopf surface. The value of 
$h\left(z_1,z_2\right)$ at the origin is 
the same nonzero value as that of 
$h\left(\lambda_1 z_1, \lambda_2 z_2\right)$
at the origin.
So $h$ must scale by $1$, i.e. $h$ must in 
fact be a holomorphic function on the Hopf 
surface, and so $h$ is a constant. So 
$f=cW_1/W_2$ is rational in $z_1$. 
Swapping the roles of $z_1$ and $z_2$
in this argument, $f$ must also be rational
in $z_2$, so a rational function.

Expanding out $W_1$ and $W_2$ in Taylor series, the terms 
$z_1^{k_1} z_2^{k_2}$ in their Taylor series 
must all scale in the same way: by a factor 
of $\lambda_1^{k_1} \lambda_2^{k_2}$. If 
there are no hyperresonances, then there can 
only be one such term, and it must be the same 
term in $W_1$ and $W_2$, so $f$ is constant. If 
there is a hyperresonance, we can multiply 
numerator and denominator Weierstrass 
polynomials each by a factor of $z_1^{k_1} z_2^{k_2}$ 
for some integers $k_1$ and $k_2$ to arrange 
that they are both rational functions of $u$.
\end{proof}

\section{Bundles on Hopf surfaces}\label{sec:VectorBundles}

\subsection{Meromorphic sections of line bundles}
If $g$ is any invertible matrix, say $N \times N$, then we 
can construct a vector bundle $\left(\C{2} \setminus 0\right) \times_{(F,g)} \C{N}$ 
over each Hopf surface $S_F$ by the equivalence $(z,v) \sim (F(z),gv)$. Conjugate 
linear maps yield isomorphic vector bundles, and splitting $g$ into Jordan blocks yields 
a sum of vector bundles, so let's assume that $g$ is a single Jordan block.
The invariant subspaces of $g$ determine a flag of invariant vector subbundles on 
the Hopf surface. The meromorphic sections are the solutions of
\[
 f(F(z))=g \, f(z).
\]
From the exponential sheaf sequence, every line bundle on every Hopf surface has 
the form $\left(\C{2} \setminus 0\right) \times_{(F,a)} \C{}$ 
for a unique nonzero complex number $a$ 
(see \cite{Barth/Peters/VanDenVen:1984} p. 226 or \cite{Mall:1991}).

\begin{proposition}[Mall \cite{Mall:1991}]\label{proposition:Mall}
Take $F : \C{2} \to \C{2}$ in the 
Poincar\'e domain, and $a \ne 0$ 
a complex number. The meromorphic 
sections of the line bundle 
$\left( \C{2} \setminus 0 \right) 
\times_{\left(F,a\right)} \C{}$ are, 
up to isomorphism (with $c$ an arbitrary complex number,  
$u=z_1^{m_1}/z_2^{m_2}$, and $P(u)$ and $Q(u)$ arbitrary polynomials):
 \[
\begin{array}{lll}\label{table:MeromorphicSections}
 F\left(z_1,z_2\right) & a \in \C{\times} & \emph{\textrm{meromorphic sections}}
 \\  \midrule 
\addlinespace[2pt]
\left(\lambda_1 z_1, \lambda_2 z_2\right), 
\lambda_1^{m_1}=\lambda_2^{m_2} & 
 \lambda_1^{k_1} \lambda_2^{k_2}
&
z_1^{k_1} z_2^{k_2} \frac{P(u)}{Q(u)}
\\
\left(\lambda_1 z_1, \lambda_2 z_2\right), \text{\normalfont{generic}} & 
 \lambda_1^{k_1} \lambda_2^{k_2} 
&
c \, z_1^{k_1} z_2^{k_2}
\\
\left(\lambda_1 z_1, \lambda_2 z_2\right) & 
a \ne \lambda_1^{k_1} \lambda_2^{k_2}
&
0
\\
\left(\lambda z_1, \lambda^m z_2 + z_1^m \right) &
 \lambda^k
&
c \, z_1^k
\\
\left(\lambda z_1, \lambda^m z_2 + z_1^m \right) &
a \ne \lambda^k
&
0
\end{array}
\]
\end{proposition}
\begin{remark}\label{remark:lineOfResonances}
If $F$ is hyperresonant 
with hyperresonance $\left(m_1,m_2\right)$, and 
$a=\lambda_1^{k_1} \lambda_2^{k_2}$, then 
draw a dot at $\left(k_1,k_2\right)$ 
and at every point given by shifting 
$\left(k_1,k_2\right)$ over by 
integer multiples of $\left(m_1,-m_2\right)$:
\begin{tikzpicture}
 \draw[very thin, gray,step=0.1] (-.05,-.05) grid (.65,.45);
\draw (-.1,.36666) -- (.6,-.1);
\draw (-.05,0) -- (.65,0);
\draw (0,-.05) -- (0,.45);
\fill (0,.3) circle (1pt); 
\fill (.3,.1) circle (1pt);
\end{tikzpicture} 
 In particular, the line through these points has 
negative slope. The Laurent series terms in each 
meromorphic section $f$ of 
$\left(\C{2} \setminus 0\right) \times_{(F,a)} \C{}$ 
have exponents $\left(j_1, j_2\right)$ 
lying on these points, so each dot represents a 
meromorphic section, up to scaling.
The holomorphic sections of the line bundle arise from the points 
in the nonnegative quadrant. The tensor product 
of line bundles is just addition of the points 
lying on the associated lines. 
\end{remark}
\begin{proof}
Take a meromorphic section, say $f$. 
Replacing $f$ with $z_1^{k_1} z_2^{k_2} f$ 
for various integers $k_1$ and $k_2$ gives 
a meromorphic section of the line bundle 
with $a$ replaced by 
$\lambda_1^{k_1} \lambda_2^{k_2} a$. Moreover, 
$f \mapsto z_1^{k_1} z_2^{k_2} f$ is an 
isomorphism of meromorphic sections of 
these line bundles. So we can arrange that 
$f$ doesn't vanish or have poles at generic 
points of both axes. But then on each axis, 
$f$ transforms like 
$f\left(\lambda_1 z_1, 0\right)
= a f\left(z_1,0\right)$. 
Clearly $f$ is meromorphic 
on both coordinate axes. Taking a Laurent 
expansion in some annulus around the origin, 
we find that $f\left(z_1,0\right)
=a_1 z_1^{\ell_1}$ for an appropriate 
choice of integer $\ell_1$, and 
$a=\lambda_1^{\ell_1}$ for some integer 
$\ell_1$. So once again replacing $f$ by 
some $z_1^{-\ell_1} f$, we can arrange 
that $f$ is a nonzero constant on the 
$z_1$ axis, and that $a=1$: $f$ is a 
meromorphic function on the Hopf surface.
Applying our classification of meromorphic
functions from lemma~\vref{lemma:MeromorphicFunctionsOnHopfSurfaces},
finally every meromorphic section 
of every line bundle 
$\left(\C{2} \setminus 0\right) \times_{\left(F,a\right)} \C{}$ 
has the form
\[
f(z)
=
z_1^{k_1} 
z_2^{k_2} 
\frac{P(u)}{Q(u)} 
\text{ where } 
u
=
z_1^{m_1}/z_2^{m_2}.
\]
By cancelling common factors, 
we can arrange that $P$ and $Q$ have no common 
zeros and that neither $P(u)$ 
nor $Q(u)$ have zeros at $u=0$.
\end{proof}


\subsection{Flat bundles of projective lines}
We can similarly consider a bundle of 
projective spaces 
$\left(\C{2} \setminus 0\right) \times_{(F,g)} \Proj{N}$, 
with $g \in \PSL{N+1,\C{}}$.  
Such bundles are
precisely the flat bundles with
projective space fibers over Hopf surfaces.
We can assume 
that $g$ is in Jordan normal form.
We will refer to any meromorphic map
$f$ on $\C{2} \setminus 0$ valued in $\Proj{N}$ 
satisfying
$f(F(z))=g \, f(z)$ as a \emph{meromorphic
section} of the bundle. For example,
if $g$ fixes infinity, we will
also equally well allow $f$ to be everywhere
infinite rather than being meromorphic,
and then also call such a map $f$
a \emph{meromorphic section}.
Obviously the constant maps 
$f$ valued in the locus of fixed 
points of $g$ will provide  
meromorphic sections. 
We will only need to consider flat bundles of 
projective lines. We will write elements of 
$\PSL{2,\C{}}$ as matrices in square brackets.
\begin{proposition}\label{proposition:MeromorphicSections}
Suppose that $F : \C{2} \to \C{2}$ 
belongs to the Poincar\'e domain and 
$g \in \PSL{2,\C{}}$. The meromorphic sections of 
the flat projective line bundle 
$\left(\C{2} \setminus 0\right) 
\times_{(F,g)} \Proj{1}$ 
are given in table~\vref{table:ProjectiveLineBundles}
(after a suitable isomorphism to put $F$ 
and $g$ into one of the indicated forms).
Every meromorphic section is a holomorphic section.
\end{proposition}
\begin{table}
\begin{center}
\[
\begin{array}{lll}
 F\left(z_1,z_2\right) & 
g \in \PSL{2,\C{}} & 
\text{meromorphic sections}
 \\  \midrule 
\addlinespace[2pt]
\left(\lambda_1 z_1, \lambda_2 z_2\right), \lambda_1^{m_1}=\lambda_2^{m_2} & 
\begin{bmatrix}
 \lambda_1^{k_1} \lambda_2^{k_2} & 0 \\
0 & 1 
\end{bmatrix}
&
z_1^{k_1} z_2^{k_2} \frac{P(u)}{Q(u)}, \infty
\\
\left(\lambda_1 z_1, \lambda_2 z_2\right), \text{\normalfont{generic}} & 
\begin{bmatrix}
 \lambda_1^{k_1} \lambda_2^{k_2} & 0 \\
0 & 1 
\end{bmatrix}
&
c \, z_1^{k_1} z_2^{k_2}, \infty
\\
\left(\lambda_1 z_1, \lambda_2 z_2\right) & 
\begin{bmatrix}
a_1 & 0 \\
0 & a_2 
\end{bmatrix}, \frac{a_1}{a_2} \ne \lambda_1^{k_1} \lambda_2^{k_2}
&
0, \infty
\\
\left(\lambda_1 z_1, \lambda_2 z_2\right) & 
\begin{bmatrix}
a & 1 \\
0 & a 
\end{bmatrix}
&
\infty
\\
\left(\lambda z_1, \lambda^m z_2 + z_1^m \right) &
\begin{bmatrix}
 \lambda^k & 0 \\
0 & 1
\end{bmatrix}
&
c \, z_1^k, \infty
\\
\left(\lambda z_1, \lambda^m z_2 + z_1^m \right) &
\begin{bmatrix}
a_1 & 0 \\
0 & a_2
\end{bmatrix}, \frac{a_1}{a_2} \ne \lambda^k
&
0, \infty
\\
\left(\lambda z_1, \lambda^m z_2 + z_1^m \right) &
\begin{bmatrix}
a & 1 \\
0 & a
\end{bmatrix},
&
\frac{z_2}{a}\left(\frac{\lambda}{z_1}\right)^m  + c, \infty
\end{array}
\]
\end{center}
\caption{The meromorphic sections of the flat 
projective line bundles 
$\left(\C{2} \setminus 0\right) \times_{(F,g)} \Proj{1}$
on Hopf surfaces, 
with $c$ an arbitrary complex number and 
$P(u)$ and $Q(u)$ arbitrary polynomials, 
and $u=z_1^{m_1}/z_2^{m_2}$.}%
\label{table:ProjectiveLineBundles}
\end{table}
The proof is split up into several lemmas.
\begin{lemma}\label{lemma:Pooh}
Take any diagonal linear map $F : \C{2} \to \C{2}$, say
\[
 F\left(z_1,z_2\right)=\left(\lambda_1 z_1, \lambda_2 z_2\right)
\]
in the Poincar\'e domain and any 
nondiagonalizable linear 
fractional transformation $g \in \PSL{2,\C{}}$.  
Then a meromorphic section 
of $\left( \C{2} \setminus 0 \right) \times_{(F,g)} \Proj{1}$
is precisely a constant mapping to the fixed 
point of $g$ on $\Proj{1}$. 
\end{lemma}
\begin{proof}
We can assume that
\[
 g =
\begin{bmatrix}
a & 1 \\
0 & a
\end{bmatrix}
\]
with $a \ne 0$. Suppose that $f$ is a meromorphic section, which we identify with a meromorphic 
function on $\C{2}$. The poles of $f\left(0,z_2\right)$ can't accumulate to $0$, unless $f$ is
infinite at the generic point of $z_1=0$. So either $f=\infty$ on $z_1=0$ or else
we can take a Laurent series expansion of $f\left(0,z_2\right)$ around $z_2=0$. 
Plug in $z_1=0$ to see that
\[
 f\left(0,\lambda^m z_2\right)=f\left(0,z_2\right)+\frac{1}{a}.
\]
The Laurent expansion has inconsistent constant term, unless $f$ is infinite 
on the line $z_1=0$. The same argument swapping the roles of
$z_1$ and $z_2$ ensures that $f$ is infinite on the line $z_2=0$.
Therefore $f$ must be infinite at all points of 
both coordinate axes. Suppose that
$f$ is not infinite everywhere. So we can write 
\[
 f = \frac{h}{z_1^{\ell_1} z_2^{\ell_2} W}
\]
with $\ell_1, \ell_2 > 0$, where $W$ is 
a Weierstrass polynomial (for one or the 
other of the axes) and $h$ is holomorphic, 
and $h$ and $W$ have no common factors 
among holomorphic functions. We will 
pick $\ell_1$ and $\ell_2$ as large as 
possible to keep $W$ holomorphic, so 
$W$ will be finite and nonzero at the 
generic point of both axes. To be a 
meromorphic section, we need
\[
 f\left(\lambda_1 z_1, \lambda_2 z_2\right)
=f\left(z_1,z_2\right)+\frac{1}{a}.
\]
In terms of the Weierstrass polynomial,
\[
 \frac{h\left(\lambda_1 z_1, \lambda_2 z_2\right)}%
{\lambda_1^{\ell_1} \lambda_2^{\ell_2} z_1^{\ell_1} z_2^{\ell_2} 
W\left(\lambda_1 z_1, \lambda_2 z_2\right)}
=
\frac{h\left(z_1,z_2\right) + 
\frac{1}{a} z_1^{\ell_1} z_2^{\ell_2} %
W\left(z_1,z_2\right)}{z_1^{\ell_1} z_2^{\ell_2} W\left(z_1,z_2\right)}.
\]
Clearly $z_1^{\ell_1} z_2^{\ell_2} 
W\left(\lambda_1 z_1, \lambda_2 z_2\right)$ 
is a Weierstrass polynomial (up to a 
constant factor) for the denominator of the left hand side, while 
$z_1^{\ell_1} z_2^{\ell_2} W\left(z_1,z_2\right)$ 
is a Weierstrass polynomial for the denominator of the right hand side. By the uniqueness of 
Weierstrass polynomials, $W$ transforms by scaling, 
so as a holomorphic section of a holomorphic line bundle, i.e. 
$W\left(\lambda_1 z_1, \lambda_2 z_2\right)
=\lambda_1^{k_1} \lambda_2^{k_2} 
W\left(z_1,z_2\right)$ for some integers $k_1$ and $k_2$. 
By remark~\ref{remark:lineOfResonances}, because $W$ is holomorphic, neither 
$k_1$ nor $k_2$ can be negative. Moreover, $h$ must transform according to
\[
 h\left(\lambda_1 z_1, \lambda_2 z_2\right)
=
\lambda_1^{k_1+\ell_1} 
\lambda_2^{k_2+\ell_2} 
\left( 
  h\left(z_1, z_2\right)+
  \frac{1}{a} z_1^{\ell_1} z_2^{\ell_2} 
  W\left(z_1,z_2\right)
\right).
\]
By proposition~\vref{proposition:Mall}, $W$ can be written as
\[
W\left(z_1,z_2\right)
=
\sum_{j=0}^{N} b_j z_1^{jm_1} z_2^{(N-j)m_2}
\]
for some complex numbers $b_0, b_1, \dots, b_N$.

In order that $W$ not vanish on either axis, we must have
\[
 W\left(z_1,z_2\right)
=
\sum_{s=0}^N b_s z_1^{sm_1} z_2^{(N-s)m_2}
\]
with $b_0 \ne 0$ and $b_N \ne 0$. In particular, 
we can take 
$\left(k_1, k_2\right)=\left(0,Nm_2\right)$. 
Expanding $h$ into a Taylor series
\[
 h\left(z_1, z_2\right)
=\sum_{n_1, n_2} a_{n_1, n_2} z_1^{n_1} z_2^{n_2},
\]
and plugging in $\left(k_1,k_2\right)=\left(0,Nm_2\right)$,
we see that the term $a_{\ell_1, \ell_2+N \, m_2}$ satisfies
\[
 \lambda_1^{\ell_1} 
  \lambda_2^{\ell_2 + N \, m_2} 
  a_{\ell_1, \ell_2+N \, m_2} = 
 \lambda_1^{\ell_1} 
  \lambda_2^{\ell_2 + N \, m_2} 
  \left( 
    a_{\ell_1, \ell_2+N \, m_2} + 
    \frac{b_0}{a} 
  \right).
\]
This forces $b_0=0$, a contradiction.
Therefore there are no meromorphic sections 
of such bundles except for $f=\infty$. 
\end{proof}
\begin{lemma}
Take any diagonal linear map $F : \C{2} \to \C{2}$, say
\[
 F\left(z_1,z_2\right)=\left(\lambda_1 z_1, \lambda_2 z_2\right)
\]
in the Poincar\'e domain and any linear 
fractional transformation $g \in \PSL{2,\C{}}$.  If
\[
g = 
\begin{bmatrix}
 \lambda_1^{k_1} \lambda_2^{k_2} & 0 \\
0 & 1 
\end{bmatrix}
\]
then the flat projective line bundle 
$\left( \C{2} \setminus 0\right) 
\times_{(F,g)} \Proj{1}$ 
has meromorphic sections
either the constant map $f\left(z_1,z_2\right)=\infty$ or
\[
f\left(z_1,z_2\right)=z_1^{k_1} z_2^{k_2} \frac{P(u)}{Q(u)} 
\]
where
\[
u=z_1^{m_1}/z_2^{m_2}
\]
and $m_1$ and $m_2$ are from the 
hyperresonance $\lambda_1^{m_1} = \lambda_2^{m_2}$. 
(We take $P(u)/Q(u)$ constant if there 
is no hyperresonance.) 
If $g$ is not conjugate in $\PSL{2,\C{}}$ 
to a matrix of the required form above, 
then a meromorphic section is precisely 
a constant mapping to one of the fixed 
points of $g$ on $\Proj{1}$. 
\end{lemma}
\begin{proof}
If $g$ is not diagonalizable, then
the result is proven in lemma~\vref{lemma:Pooh}.
So assume that $g$ is diagonal, say
\[
 g=
\begin{bmatrix}
 a_1 & 0 \\
0 & a_2
\end{bmatrix}
\]
and the contraction $F$ is diagonal linear.
Then a meromorphic section is a meromorphic 
function $f : \C{2} \setminus 0 \to \C{}$ for which 
$f\left(\lambda_1 z_1, \lambda_2 z_2\right)=
\frac{a_1}{a_2}f\left(z_1,z_2\right)$. 
Again, such a function $f$ has the form
\[
 f\left(z_1,z_2\right)
=z_1^{k_1} z_2^{k_2} 
\frac{P(u)}{Q(u)} 
\text{ where } u=z_1^{m_1}/z_2^{m_2}
\]
and $a_1/a_2$ must have the form $\lambda_1^{k_1} \lambda_2^{k_2}$ 
or else $f$ is constant and equal to $0$ or $\infty$. 
Therefore
\[
 g=
\begin{bmatrix}
 \lambda_1^{k_1} \lambda_2^{k_2} & 0 \\
0 & 1
\end{bmatrix}
\]
up to rescaling, or else there are no meromorphic sections
other than $f=0$ and $f=\infty$.
\end{proof}

\begin{corollary}
Every meromorphic section of a flat projective 
line bundle on any Hopf surface is a holomorphic section.
\end{corollary}
\begin{proof}
In general, a meromorphic function $f$ on a complex surface need not be a 
holomorphic map to $\Proj{1}$. It will be a holomorphic map to $\Proj{1}$ just 
when, near each point, it can be made a holomorphic function by a linear fractional 
transformation. Equivalently, either $f$ or $1/f$ is a holomorphic function 
at each point. Equivalently, either $f$ is finite, or $1/f$ is finite near each point. 
Equivalently, the zero locus of $f$ does not cross the poles of $f$. Every meromorphic 
section of any flat bundle of projective lines over a diagonal Hopf surface is holomorphic, 
as the zeroes and poles occur only along the curves $z_1=0$, $z_2=0$ and 
$z_1^{m_1}=\left(\text{constant}\right)z_2^{m_2}$. 
\end{proof}

\begin{lemma}
 Consider an exceptional map $F$ in the Poincar\'e domain. We can assume that 
$F\left(z_1,z_2\right)=\left(\lambda z_1, \lambda^m z_2 + z_1^m\right)$. 
Take a linear fractional transformation $g$ fixing a single point of $\Proj{1}$. We can assume that 
\[
 g =
\begin{bmatrix}
a & 1 \\
0 & a 
\end{bmatrix}.
\]
Then the meromorphic sections of the $\Proj{1}$-bundle $\left(\C{2} \setminus 0\right) 
\times_{\left(F,g\right)} \Proj{1}$ are precisely the functions
\[
f = \frac{z_2}{a} \left(\frac{\lambda}{z_1}\right)^m + b,
\]
for any constant $b$. All meromorphic sections are holomorphic.
\end{lemma}
\begin{proof}
 Suppose that $f$ is a meromorphic section, which we identify with a meromorphic 
function on $\C{2}$. The poles of $f\left(0,z_2\right)$ can't accumulate to $0$, unless $f$ is
infinite at the generic point of $z_1=0$. So either $f=\infty$ on $z_1=0$ or else
we can take a Laurent series expansion of $f\left(0,z_2\right)$ around $z_2=0$. 
Plug in $z_1=0$ to see that
\[
 f\left(0,\lambda^m z_2\right)=f\left(0,z_2\right)+\frac{1}{a}.
\]
The Laurent expansion has inconsistent constant term, unless $f$ is infinite 
on the line $z_1=0$. We can therefore write
\[
 f(z) = \frac{h(z)}{z_1^{\ell} W(z)}
\]
for some $\ell>0$, with $h(z)$ holomorphic and $W(z)$ a Weierstrass polynomial in $z_1$ 
not dividing into $h(z)$. Then
\[
 f(F(z))=\frac{h(F(z))}{\lambda^{\ell} z_1^{\ell} W(F(z))}
=\frac{h(z) + \frac{1}{a} \, z_1^{\ell} W(z)}{z_1^{\ell} W(z)}.
\]
Clearly (up to scaling by a constant) $W(z)$ 
is a Weierstrass polynomial for the denominator 
of $z_1^{\ell} f(F(z))$, as is $W(F(z))$, 
along the line $z_1=0$. By uniqueness of 
Weierstrass polynomials, $W(F(z))=c \, W(z)$ 
for some constant $c$. But then $W$ must be a 
holomorphic section of a line bundle on the 
associated Hopf surface.
By proposition~\vref{proposition:Mall},
$W\left(z_1,z_2\right)=z_1^k$ for some integer 
$k \ge 0$. So we can assume that
\[
 f(z)=\frac{h(z)}{z_1^{\ell}}
\]
with $\ell > 0$. Therefore 
\[
 h\left(\lambda z_1, \lambda^m z_2 + z_1^m\right)
=\lambda^{\ell} \left( h\left(z_1,z_2\right) 
+  \frac{z_1^{\ell}}{a} \right).
\]
Again restrict to $z_1=0$ to see that
\[
 h\left(0, \lambda^m z_2 \right)
=\lambda^{\ell} h\left(0,z_2\right). 
\]
Expand $h$ in a Taylor series to see 
that $\ell=km$ for some integer $k \ge 0$ and
\(
 h\left(0,z_2\right)=\frac{z_2^k}{a}.
\)
So we see that
\[
 h\left(\lambda z_1, \lambda^m z_2 + z_1^m\right)
=\lambda^{km} \left( h\left(z_1,z_2\right) 
+ \frac{z_1^{km}}{a} \right).
\]
Following Kodaira \cite{Kodaira:1966} p. 697 equation 100, 
we let $h_1=\pd{h}{z_2}$. Then we calculate that
\[
 h_1\left(F(z)\right) = \lambda^{(k-1)m} h_1(z).
\]
Therefore $h_1$ is a holomorphic section of a line bundle,
and by proposition~\vref{proposition:Mall},
\[
 h_1\left(z_1, z_2\right)=c \, z_1^{(k-1)m},
\]
so that
\[
 h\left(z_1,z_2\right)=c \, z_1^{(k-1)m} z_2 + \sum_{s=0}^{\infty} a_s z_1^s,
\]
for some complex numbers $a_s$, and we plug in to find
\begin{align*}
 0 &=
h\left(\lambda z_1, \lambda^m z_2 + z_1^m \right)
-\lambda^{km} \left( h\left(z_1,z_2\right) +  \frac{z_1^{km}}{a}\right)
\\
&=
\sum_{s=0}^{m}
\left(
\left(\lambda^s - \lambda^{km}\right) a_s
+ 
\delta_{s=km}\left(c-\frac{\lambda^m}{a}\right) \lambda^{(k-1)m}
\right) z_1^s.
\end{align*}
Looking at the $z_1^{km}$ coefficient yields $c=\lambda^m/a$. 
Plugging in $s\ne km$ yields $a_s=0$. Therefore
\[
 h\left(z_1, z_2\right)
=
\frac{\lambda^m \, z_1^{(k-1)m} z_2}{a} + b \, z_1^{km}.
\]
Finally
\[
 f\left(z_1,z_2\right)=\frac{\lambda^m z_2}{a z_1^m} + b.
\]
\end{proof}

\section{Geometric structures on Hopf surfaces}\label{sec:GeometricStructures}

\subsection{Geometric structures}

If $G/H$ is any homogeneous space, 
a $G/H$-structure is a maximal atlas of 
charts valued in $G/H$, with transition 
maps in $G$. The identity map of a homogeneous
space $G/H$ is contained in a unique
$G/H$-structure, and 
is called the \emph{model} $G/H$-structure.
If $G/H$ is affine 
(projective) space and $G$ is the group 
of affine (projective) transformations, 
then $G/H$-structures are called 
\emph{affine (projective) structures}. 
Clearly any linear Hopf surface $S_F$ 
bears an affine structure, since the 
transition map $F$ is linear. 

\begin{example}
For example, let
$G=\GL{2,\C{}}$ and $H$ the subgroup
fixing the point $(1,0) \in \C{2}$.
So $G/H=\C{2} \setminus 0$. Every
linear Hopf surface has an obvious
$G/H$-structure, since its universal
covering space is $G/H$ and its
covering group acts by an element of $G$.
\end{example}

\begin{example}
 If $G=\PSL{n+1,\C{}}$ and $H$ is the stabilizer
of a point of $\Proj{n}$, then a $G/H$-structure
is called a \emph{projective connection}.
\end{example}

\subsection{Developing maps and holonomy}

\begin{lemma}
Every $G/H$-structure on any manifold $M$ is obtained from a local diffeomorphism 
$\dev : \tilde{M} \to G/H$ of the universal covering space of $M$ 
(called the \emph{developing map}), equivariant for a homomorphism 
$\hol : \pi_1\left(M\right) \to G$ (called the \emph{holonomy}):
to recover the $G/H$-structure,
compose $\dev$ with a local inverse of $\tilde{M} \to M$
to give an atlas of local coordinates valued in $G/H$.

The pair $\left(\dev, \hol\right)$ are only defined up to 
the $G$-action $\left(g \, \dev, g \, \hol \, g^{-1}\right)$. 
Conversely, the $G$-orbit of this pair under this action determines the $G/H$-structure. 
Moreover, any choice of two maps 
$\dev : \tilde{M} \to G/H$ and $\hol : \pi_1\left(M\right) \to G$
with $\hol$ a group morphism and $\dev$ a $\hol$-equivariant
local diffeomorphism determines a unique $G/H$-structure on $M$.
\end{lemma}
\begin{proof}
 See Thurston \cite{Thurston:1997} p. 140.
\end{proof}

Therefore we will classify $G/H$-structures on Hopf surfaces
by writing out their developing maps and holonomies.

\begin{definition}
If $M$ is a complex manifold and $G$ is a complex Lie group with
$H$ a closed complex subgroup, then a $G/H$-structure is called
\emph{holomorphic} if its developing map is holomorphic,
or equivalently if all of the charts of the structure are
holomorphic maps. From now on all locally homogeneous structures will be assumed
holomorphic.
\end{definition}

\begin{definition}
A structure is called \emph{complete} if the developing map is onto. 
\end{definition}

\begin{definition}
A structure is called \emph{essential} if the developing map is injective.
\end{definition}

Essential structures are precisely the induced structures
on manifolds covered by open sets of the model. For example, a Riemann surface of any
genus has the obvious holomorphic projective connection given by
the inclusion $\Delta \subset \C{} \subset \Proj{1}$.

\begin{definition}
On a Hopf surface $S_F$, the fundamental group has the distinguished generator $F$, 
so the holonomy map $\hol$ is determined by the element $\hol(F) \in G$; 
refer to this element as the \emph{holonomy generator} of the $G/H$-structure.
\end{definition}

\begin{definition}
A \emph{branched $G/H$-structure} on a manifold $M$ is a choice of map 
$\tilde{M} \to G/H$ (again called the \emph{developing map}), equivariant for a 
homomorphism $\pi_1(M) \to G$ (again called 
the \emph{holonomy}), determined up to 
the same $G$-action. 
\end{definition}
The developing map of a branched structure might \emph{not} 
be a local biholomorphism. The basic difficulty we encounter in this paper is 
that of distinguishing branched from unbranched structures. Obstructions to 
structures are usually also obstructions to branched structures, so if there is 
a branched structure, then most of the obstructions we can come up with will 
not help us to rule out the possibility of an unbranched structure.

\subsection{Induction of structures}

\begin{definition}
Suppose that $G_0/H_0$ and $G/H$ are homogeneous spaces. 
A \emph{morphism} of homogeneous spaces $\Phi : G_0/H_0 \to G/H$ means a
morphism $\Phi : G_0 \to G$ of Lie groups so that $\Phi\left(H_0\right) \subset H$.
We will also denote the induced map $G_0/H_0 \to G/H$ by the letter $\Phi$.
A morphism of homogeneous spaces is called an \emph{avatar} if the induced smooth map $G_0/H_0 \to G/H$
is a local diffeomorphism.  
\end{definition}

\begin{definition}
If $\Phi : G_0/H_0 \to G/H$ is an avatar, and we have a $G_0/H_0$-structure on a manifold $M$,
with developing map $\dev_0 : \tilde{M} \to G_0/H_0$ and holonomy
$\hol_0 : \pi_1(M) \to G_0$, then the \emph{induced} $G/H$-structure is the one
given by $\dev = \Phi \circ \dev_0$ and $\hol = \Phi \circ \hol_0$.
\end{definition}

\section{The model}\label{section:TheModel}

\subsection{Definition}

As usual, we treat points of $\Proj{1}$ as lines through $0$ in $\C{2}$, and we 
write $\OO{n}$ for the bundle over $\Proj{1}$ whose fiber over a point 
$L \in \Proj{1}$ is the $n$-fold symmetric product $\Sym{n}{L}^*$, if $n>0$, and $\Sym{|n|}{L}$ if $n<0$, 
and $\C{}$ if $n=0$. We will denote the total space of the
bundle $\OO{n}$ also as $\OO{n}$. Clearly $\OO{n}$ is a complex surface.
For now, let's assume that $n>0$ and write the points of $\OO{n}$ 
as pairs $(L,q)$ with $q \in \Sym{n}{L}^*$. Thus the global sections of $\OO{n} \to \Proj{1}$
are the homogeneous polynomials of degree $n$, $\Sym{n}{\C{2}}^*$. Let 
\[
G=\Gn 
\]
act on $\OO{n}$ by $(g,p)(L,q)=\left(gL,q \, g^{-1} + \left.p\right|_{gL}\right)$. 
The multiplication in $G$ is $\left(g_0,p_0\right)\left(g_1,p_1\right)
=\left(g_0g_1,p_0+p_1 g_0^{-1}\right)$ and the inverse operation is 
$\left(g,p\right)^{-1}=\left(g^{-1},-pg\right)$. Let $H$ be the stabilizer of 
$\left(L_0,0\right)$, where $L_0$ is the line $z_2=0$ in $\C{2}$; i.e.
\[
H=\Hn.
\]
Clearly $\OO{n}=G/H$. Moreover, $G$ acts freely on $\OO{n}$.
The action of $G$ preserves the fiber bundle map $\OO{n} \to \Proj{1}$, and preserves 
the affine structure on each fiber. It also acts transitively on the 
global sections of $\OO{n} \to \Proj{1}$, a family of curves transverse to the 
fibers. Moreover it acts on $\Proj{1}$ via a surjection to the group of linear fractional
transformations. 

Every surface with $\OO{n}$-structure inherits a foliation, corresponding to 
the fiber bundle map, and inherits a family of curves which locally are 
identified with the global sections. Locally, on open sets on which the
foliation is a fibration, the base space of the fibration is a Riemann
surface with projective structure; in this sense the foliation has a 
\emph{transverse projective structure}.

\subsection{Ordinary differential equations and geometric structures}

It is well known (see Lagrange \cite{Lagrange:1957a,Lagrange:1957b}, Fels \cite{Fels:1993,Fels:1995}, 
Dunajski and Tod \cite{Dunajski/Tod:2006}, Godli{\'n}ski and Nurowski \cite{Godlinski/Nurowski:2007}, 
Doubrov \cite{Doubrov:2008}) that every (real or holomorphic) scalar ordinary differential equation 
of order $n+1\ge 3$ has a symmetry Lie algebra of point transformations of dimension at most $n+5$, and this 
dimension is acheived just precisely for the ordinary differential equations which are locally 
identified by point transformation with the equation $\frac{d^{n+1}y}{dx^{n+1}}=0$. 
Moreover, every holomorphic scalar ordinary differential equation locally point equivalent to 
$\frac{d^{n+1}y}{dx^{n+1}}=0$ is locally determined by, and locally determines, an 
$\OO{n}$-structure. Each solution of the differential equation is identified by the developing 
map with a global section of $\OO{n}$. 

\subsection{Conjugacy classes in the symmetry group}

Recall that
\[
G=\Gn.
\]
In this section we show that every element $(g,p) \in G$ is 
conjugate to one of a certain normal form defined in definition~\vref{definition:NormalForm}.

%
The conjugates of an element $(g,p) \in G$ are the elements of the form
\[
 \left(g_0 g g_0^{-1}, p g_0^{-1} + p_0 - p_0 g_0 g^{-1} g_0^{-1}\right).
\]
\begin{lemma}\label{lemma:InvariantPolynomials}
Pick a matrix $g \in \GL{2,\C{}}$. There are no nonzero $g$-invariant 
homogeneous polynomials of degree $n$ just when, for all homogeneous 
polynomials $p$ of degree $n$, $(g,p)$ is conjugate to $(g,0)$.
\end{lemma}
\begin{proof}
We can take $g_0=I$, and then we have to solve $p_0 g^{-1} - p_0 = p$. 
The kernel of the map $p_0 \mapsto p_0 g^{-1} - p_0$ 
is precisely the $g$ invariant homogeneous polynomials of degree $n$. 
Therefore the linear map $p_0 \mapsto p_0 g^{-1} - p_0$ is 
onto just when there are no $g$-invariant polynomials. 
\end{proof} 
\begin{corollary}
If $g \in \GL{2,\C{}}$ lies in the Poincar\'e domain, or if $g^{-1}$ does, then, 
for any homogeneous polynomial $p$ of any positive degree, $(g,p)$ is conjugate to $(g,0)$ 
\end{corollary}
\begin{lemma}\label{lemma:notDiagonalHolonomy}
If $g$ is not diagonalizable, then, for any homogeneous polynomial $p$ 
of any positive degree, either (1) $(g,p)$ is conjugate to $(g,0)$ or
(2) $(g,p)$ is conjugate to 
\[
\left(
\begin{pmatrix}
 1 & 1 \\
0 & 1
\end{pmatrix}, Z_1^n\right).
\]
\end{lemma}
\begin{proof}
Suppose that 
\[
 g = \begin{pmatrix}
      a & 1 \\
      0 & a
     \end{pmatrix}
\]
and that $(g,p)$ is not conjugate to $(g,0)$. By lemma~\vref{lemma:InvariantPolynomials},
we can take a nonzero $g$-invariant polynomial $p_0$ 
of degree $n$. Factor out as many factors of $z_2$ from $p_0\left(z_1,z_2\right)$ as possible, say
\[
 p_0\left(z_1,z_2\right)=q_0\left(z_1,z_2\right)z_2^k.
\]
Then $q_0$ must scale under 
$g$-action by $\frac{1}{a^{k}}$. 
Suppose that $q_0$ has degree $m$. Write out coefficients
\[
q_0\left(z_1,z_2\right) = \sum b_i z_1^i z_2^{m-i}.
\]
The highest order term in $z_1$ 
must be $b_{n} z_1^{m} \ne 0$, 
since otherwise we could factor 
out more factors of $z_2$ from $q_0$. Compute out
\[
 q_0\left(gz\right)=
b_m a^m z_1^m + \left(m b_m a^{m-1} + b_{m-1} a^m \right) z_1^{m-1} z_2 + \dots
\]
In order that $q_0$ scale 
by $\frac{1}{a^{k}}$, we must have
\begin{align*}
b_m a^m &= \frac{b_m}{a^k} \\
m b_m a^{m-1} + b_{m-1} a^m &= \frac{b_{m-1}}{a^k}.
\end{align*}
Since $b_m \ne 0$, we must have 
$a^{m+k}=1$, a root of unity. But then
the second equation becomes
\( m b_m = 0. \)
Since $b_m \ne 0$, we must have $m=0$, so $q_0$ is constant and
$p_0=c \, z_2^n$. 

Because $p_0$ is $g$-invariant, we must have $a^n=1$. Since
$(g,p) \in G$ has matrix part $g$ defined
only up to multiplication by $n$-th roots of 1,
we can arrange $a=1$. The polynomials
that we can arrive at in the form
$p_0 - p_0 g^{-1}$ are clearly precisely
those of the form
\[
\sum_{j=1}^{n}
b_j 
\left( 
\binom{j}{1} z_1^{j-1} z_2^{n-j+1} 
-\binom{j}{2} z_1^{j-2} z_2^{n-j+2} + 
\dots
+
(-1)^j \binom{j}{j-1} z_1 z_2^{n-1}
+ z_2^n
\right). 
\]
Looking at the leading terms in $z_1$,
we see that we can successively
pick $b_1, b_2, \dots$ to kill
off the $z_1^{n-1} z_2, z_1^{n-2} z_2^2, \dots$ terms
in $p$ by conjugation by
$\left(I,p_0\right)$, until
we kill off all terms except
the $z_1^n$ term. Then we
rescale by conjugation
by $\left(\lambda I,0\right)$
to rescale $p$ as needed to
arrange $p=z_1^n$.
\end{proof}

\begin{definition}
A matrix $g \in \GL{2,\C{}}$ is 
called \emph{hyperresonant} if 
it is diagonalizable with eigenvalues 
$\lambda_1, \lambda_2$ satisfying 
$\lambda_1^{m_1} = \lambda_2^{m_2}$
for some pair of integers $\left(m_1,m_2\right) \ne (0,0)$.
Such a pair of integers $\left(m_1,m_2\right)$ 
will be called a \emph{hyperresonance pair} of $g$,
and the collection of hyperresonance pairs
(an abelian subgroup of $\Z{2}$)
will be called the \emph{hyperresonance group} $\Lambda_g$ of $g$.
If the hyperresonance group is of
rank 1, we take the element $\left(m_1,m_2\right)$
with smallest positive $m_1$ (or $\left(0,m_2\right)$
with smallest positive $m_2$), and
call it the \emph{hyperresonance} of $g$.
\end{definition}

\begin{lemma}\label{lemma:HyperresonanceGroupRank}
The hyperresonance group of a diagonalizable
matrix $g$ with eigenvalues $\lambda_1, \lambda_2$
has rank 0 just when $g$ is not hyperresonant,
rank 2 just when $\lambda_1 = e^{2 \pi i p_1/q_1}$
and
$\lambda_2 = e^{2 \pi i p_2/q_2}$
where $\frac{p_1}{q_1}$ and $\frac{p_2}{q_2}$
are rational numbers,
and rank 1 otherwise.
\end{lemma}
\begin{proof}
Suppose that
$\lambda_1 = e^{2 \pi i p_1/q_1}$
and
$\lambda_2 = e^{2 \pi i p_2/q_2}$
where $\frac{p_1}{q_1}$ and $\frac{p_2}{q_2}$
are rational numbers.
Clearly $\left(q_1,0\right)$
and $\left(0,q_2\right)$ lie in $\Lambda_g$,
so $\Lambda_g$ has rank 2.

Conversely, suppose that
the rank of $\Lambda_g$ is 2.
Suppose that $\lambda_1^{m_1} = \lambda_2^{m_2}$
with $\left(m_1,m_2\right) \ne (0,0)$.
Let $r_j = \log \left|\lambda_j\right|$.
Then $m_1 r_1 = m_2 r_2$. So the hyperresonance
group lies on a line through $0$ in $\R{2}$,
unless $r_1=r_2=0$. If $r_1 \ne 0$ or
$r_2 \ne 0$, then there is an
integer point on that line with smallest 
nonzero distance from the origin,
and $\Lambda_g$ is of rank 1.

So we can suppose that $r_1=r_2=0$,
i.e. $\lambda_1 = e^{2 \pi i a_1}$
$\lambda_2 = e^{2 \pi i a_2}$
for some real numbers $0 \le a_1, a_2 < 1$.
The hyperresonance pairs are just the
pairs of integers $\left(m_1,m_2\right)$
for which $m_1 a_1 + m_1 a_2$
is an integer. The
hyperresonance group spans $\R{2}$,
since it doesn't lie on a line.
Take
two linearly independent hyperresonances $\left(m_1,m_2\right)$
and $\left(n_1,n_2\right)$. Then
\[
\begin{pmatrix}
m_1 & m_2 \\
n_1 & n_2 
\end{pmatrix}
\begin{pmatrix}
a_1 \\
a_2
\end{pmatrix}
=
\begin{pmatrix}
b_1 \\
b_2
\end{pmatrix}
\]
where $b_1, b_2$ are integers.
Therefore
\[
 \begin{pmatrix}
a_1 \\
a_2
\end{pmatrix}
=
\begin{pmatrix}
m_1 & m_2 \\
n_1 & n_2 
\end{pmatrix}^{-1}
\begin{pmatrix}
b_1 \\
b_2
\end{pmatrix}
\]
are rational numbers, say $a_j=p_j/q_j$
with $p_j,q_j$ integers, $j=1,2$. 
\end{proof}


\begin{definition}
Suppose that $(g,p)\in G$
and that $g$ is diagonalizable.
The \emph{resonant degrees} of $p$ are the
integers $k$ with $0 \le k \le n$ for which
the eigenvalues $\lambda_1, \lambda_2$ of $g$
satisfy $\lambda_1^{k} \lambda_2^{n-k}=1$.
If we write 
\[
p\left(Z_1,Z_2\right)=
\sum_{k=0}^{n} a_k Z_1^k Z_2^{n-k},
\]
the \emph{resonant terms} of $p$ are the
terms $a_k Z_1^{k} Z_2^{n-k}$
for which $k$ is a resonant degree.
The \emph{leading resonant term} is the
term $a_k z_1^{k} z_2^{n-k}$ with smallest
resonant degree $k$ for which $a_k \ne 0$.
The \emph{trailing resonant term} is the
term $a_k z_1^{k} z_2^{n-k}$ with largest
resonant degree $k$ for which $a_k \ne 0$.
\end{definition}

\begin{definition}
Suppose that $(g,p)\in G$ and that
$g$ is not diagonalizable.
We will declare that $n$ is a 
\emph{resonant degree} of $p$ if $g$
has eigenvalue $\lambda$
an $n$-th root of 1,
and declare that there are no resonant
terms otherwise. Once again, the resonant
terms are the nonzero terms of 
resonant degree.
\end{definition}

\begin{definition}\label{definition:NormalForm}
We will say that an element $(g,p) \in G$ is in \emph{normal form} if 
either
\begin{enumerate}
 \item 
\[
 g=
\begin{pmatrix}
 \lambda & 1 \\
0 & \lambda
\end{pmatrix}
\]
and $p=0$
or 
\item
\[
 g=
\begin{pmatrix}
1 & 1 \\
0 & 1
\end{pmatrix}
\]
and $p\left(Z_1,Z_2\right)=Z_1^n$
or 
\item
\[
 g=
\begin{pmatrix}
 \lambda_1 & 0 \\
0 & \lambda_2
\end{pmatrix}
\]
is diagonal and 
\[
p\left(Z_1,Z_2\right) = 
\sum_{k=0}^n a_k Z_1^{k} Z_2^{n-k}
\]
with all nonresonant terms vanishing,
and $a_k=1$ for both the leading
and trailing resonant terms.
\end{enumerate}
\end{definition}

\begin{lemma}\label{lemma:holonomyNormalForm}
Every element $(g,p) \in G$ 
is conjugate to an element 
in normal form.
Either (1) the normal form is 
unique up to possibly permuting coordinates $z_1$ and $z_2$
or (2) $g=I$ and all terms of $p$ are resonant.

Suppose that there are at least two distinct resonant 
terms. Then (1) $g$ is diagonalizable with eigenvalues
$\lambda_1=e^{2 \pi i p_1/q_1}$ and
$\lambda_2=e^{2 \pi i p_2/q_2}$
with $\frac{p_1}{q_1}, \frac{p_2}{q_2}$
rational numbers and (2) for each resonant degree $k$,
\[
k \frac{p_1}{q_1} + \left(n-k\right)\frac{p_2}{q_2}
\]
is an integer.
\end{lemma}
\begin{proof}
By lemma~\vref{lemma:notDiagonalHolonomy}, 
we can assume that 
\[
g 
=
\begin{pmatrix}
\lambda_1 & 0 \\
0 & \lambda_2
\end{pmatrix}.
\]
Suppose that
\[
p\left(Z_1,Z_2\right)
=
\sum_k a_k Z_1^k Z_2^{n-k}.
\]
Then pick any element
$\left(g_0,p_0\right) \in G$ of the form
\[
g_0 
=
\begin{pmatrix}
\mu_1 & 0 \\
0 & \mu_2
\end{pmatrix}
\]
and say
\[
p_0\left(Z_1,Z_2\right)
=
\sum_k b_k Z_1^k Z_2^{n-k}.
\]
Define a polynomial $p_1$ by
\[
\left(g,p_1\right)=\left(g_0,p_0\right)
(g,p)
\left(g_0,p_0\right)^{-1}.
\]
i.e.
\[
p_1
=
p g_0^{-1} + p_0 - p_0 g^{-1}.
\]
Therefore the coefficients of $p_1$ 
are 
\[
\frac{a_k}{\mu_1^{k} \mu_2^{n-k}} 
+ b_k \left(1 - \frac{1}{\lambda_1^{k} \lambda_2^{n-k}}\right).
\]
So we can conjugate $(g,p)$ 
to arrange $a_k=0$ by choice of $b_k$ 
unless $k$ is a resonance degree.
So we can and will assume that 
$p$ has only resonant terms.
If there is exactly one resonant
term, say degree $k$, then we can arrange 
by choice of 
the coefficients $\mu_1$ and $\mu_2$
that $p\left(Z_1,Z_2\right) = Z_1^k Z_2^{n-k}$.

If $p$ has two or more resonant terms, then
we can find two corresponding resonant degrees, say $k_1$ and $k_2$.
The hyperresonant pairs 
$\left(k_1,n-k_1\right)$ and $\left(k_2,n-k_2\right)$
are linearly independent elements
of the hyperresonance group, which must therefore 
have rank 2. By lemma~\vref{lemma:HyperresonanceGroupRank}, 
$g$ has eigenvalues $\lambda_1=e^{2 \pi i p_1/q_1}$ and
$\lambda_2=e^{2 \pi i p_2/q_2}$
with $\frac{p_1}{q_1}, \frac{p_2}{q_2}$
rational numbers, and
\[
k \frac{p_1}{q_1} + \left(n-k\right)\frac{p_2}{q_2} 
\]
is an integer for each resonant degree $k$.

Next we can pick $\mu_1$ and $\mu_2$
to arrange that the leading and trailing
resonant coefficients are $1$. So 
we have acheived normal form.
If the eigenvalues $\lambda_1$
and $\lambda_2$ are not equal,
then the only choices of $g_0$
which will preserve the diagonalization
of $g$ are diagonal themselves, and
we immediately see that there
is a unique normal form up to 
swapping the coordinates $Z_1$ and $Z_2$.

If the normal form is not unique,
then $g=\lambda I$, for some constant
$\lambda \in \C{\times}$. If $\lambda$
is not an $n^{\text{th}}$ root of 1,
then there are no $g$ invariant homogeneous
polynomials, so we can arrange after
conjugation $(g,p)=\left( \lambda I, 0\right)$,
arriving at normal form. Suppose that
$\lambda$ is an $n^{\text{th}}$ root of 1.
But $(\lambda I,p)=(I,p) \in G$,
since we mod out by $n^{\text{th}}$ 
roots of 1. So we can assume that $\lambda=1$,
i.e. $g=I$.
Under conjugation, $p$ is acted on by $g_0$. 
The roots of $p$, with multiplicities,
are transformed by linear fractional
transformation. Since $g_0$ can rescale $p$,
the conjugacy classes of elements of $G$
of the form $(I,p)$ are precisely identified
with choices of $n$ unordered points on
$\Proj{1}$, not necessarily distinct, modulo
linear fractional transformations of $\Proj{1}$.
For more on these conjugacy classes,
see Popov and Vinberg \cite{Shafarevich:1994} p. 140
or Howard et. al. \cite{Howard/Millson/Snowden/Vakil:2009}.
\end{proof}

\begin{definition}
We will say that an element 
$(g,p) \in G$ is 
\emph{generic} if it is conjugate 
to an element of the form $(g',0)$ where 
\[
 g'=
\begin{pmatrix}
 a_1 & 0 \\
0 & a_2
\end{pmatrix}
\]
is diagonal (with the same eigenvalues as $g$). 
\end{definition}
Generic elements form a dense open subset of $G$. 
An element $(g,p)$ is not generic just when 
(1) $g$ is not diagonalizable or 
(2) $g$ is hyperresonant and $p$ has 
at least one resonant degree.

\subsection{Affine coordinates}

We want to cover $\OO{n}$ in coordinate charts,
which we will refer to as \emph{affine coordinates}
on $\OO{n}$.
Take a line $L_0$ in $\C{2}$, which we will
also think of as a point of $\Proj{1}$.
Consider the open subset of 
$\OO{n}$ which lies over $\C{}=\Proj{1} \setminus L_0$.
This open subset of $\OO{n}$ is preserved 
by all of the elements
$(g,p) \in G$ for which $g$ 
fixes the line $L_0$.
Let's use linear coordinates $Z_1, Z_2$ on $\C{2}$.
Take $L_0$ to be the line $Z_2=0$. 
We will now produce coordinates $t_1, t_2$ on 
the corresponding open subset of $\OO{n}$. 
Map $\left(t_1,t_2\right) \in \C{2} \mapsto 
\left(L,q\right) \in \OO{n}$ where 
$L$ is the line $Z_1 = t_1 Z_2$, and 
$q=t_2 \left.Z_2^n\right|_{L}$. Clearly 
$t_1=Z_1/Z_2$ is an 
affine chart on $\Proj{1}$. In these coordinates, 
if we let
\[
 g=
\begin{pmatrix}
a & b \\
c & d
\end{pmatrix}
\]
then the element $(g,0) \in G$ acts by 
\[
(g,0)\left(t_1,t_2\right)=
\left(
\frac{at_1+b}{ct_1 + d},
\frac{t_2}{\left(ct_1+d\right)^n}
\right).
\]
If $p\left(Z_1,Z_2\right)=\sum_{i+j=n} a_{ij} Z_1^i Z_2^j$, then
the element $(I,p) \in G$ acts by 
\begin{align*}
 (I,p)\left(t_1,t_2\right)&=
\left(
t_1
,
t_2+p\left(t_1,1\right)
\right)
\\
&=
\left(
t_1
,
t_2+\sum_{i+j=n} a_{ij} t_1^i
\right).
\end{align*}
We cover $\OO{n}$ in two coordinate 
charts: $\left(t_1,t_2\right)$ and 
\[
\left(s_1,s_2\right)=
\left(\frac{1}{t_1},\frac{t_2}{t_1^n}\right)
=
g\left(t_1,t_2\right).
\]
where
\[
 g=\begin{pmatrix}
    0 & 1 \\
    1 & 0
   \end{pmatrix}.
\]

The global sections $p\left(Z_1,Z_2\right)$ 
of $\OO{n}$, i.e. homogeneous polynomials
of degree $n$,
when written in these coordinates
become $t_2=p\left(t_1,1\right)$.
In particular, they satisfy 
\[
\frac{d^{n+1} t_2}{dt_1^{n+1}}=0.
\]

%
%
%

\section{Examples}\label{section:Examples}

Our examples are summarized 
in tables~\vref{table:examples} and \vref{table:exoticStructures}.  We will now explain them
in detail.

\begin{table}
\begin{center}
\[
\begin{array}{rccc}
\OO{n}\text{-structure} & 
F & 
\text{developing map} & 
\text{holonomy generator}
\\ \midrule
\text{radial} & 
\text{linear} &
\left(\frac{z_1}{z_2},\frac{1}{z_2^n}\right) &
(F,0) 
\\
\text{eigenstructure} & 
\begin{pmatrix}
\lambda_1 & 0 \\
0 & \lambda_2
\end{pmatrix}
 &
\left(z_1,z_2\right) &
\left(
\begin{pmatrix}
\frac{\lambda_1}{\lambda_2^{1/n}} & 0 \\
0 & \frac{1}{\lambda_2^{1/n}}
\end{pmatrix},
0\right)
\\
\text{eigenstructure} & 
\left(\lambda z_1,\lambda^m z_2 + z_1^m\right) &
\left(z_1,z_2\right) &
\left(
\begin{pmatrix}
\frac{\lambda}{\lambda^{m/n}} & 0  \\
0 & \frac{1}{\lambda^{m/n}}
\end{pmatrix}
,
\frac{1}{\lambda^m} Z_1^m Z_2^{n-m}
\right)
\\
\text{hyperresonant} & 
\text{see table~\ref{table:exoticStructures}} 
&
\text{see table~\ref{table:exoticStructures}}
&
\text{see table~\ref{table:exoticStructures}}
\end{array}
\]
\end{center}
\caption{Examples
of structures on each Hopf surface $S_F$, expressed
in affine coordinates.}\label{table:examples}
\end{table}

\subsection{The radial structures on linear Hopf surfaces}

We can map $\OO{n} \to \OO{kn}$ by 
$\left(L,q\right) \mapsto \left(L,q^k\right)$ 
for any integer $k>0$. Similarly, we can map 
$\OO{n} \setminus 0 \to \OO{kn} \setminus 0$ 
for all integer values of $k$: if $k<0$ then 
take any element $(L,q)$ with $q \ne 0$ to 
the pair $(L,r^{|k|})$ where $r$ is dual 
to $q$ in $\Sym{n}{L}$. These maps are 
local biholomorphisms away from the 
$0$-sections. Moreover, these maps are 
equivariant under $\GL{2,\C{}}$. In 
particular, 
$\OO{-1} \setminus 0
=\C{2} \setminus 0$ maps by local 
biholomorphism to $\OO{n}$ for every
$n \ne 0$. The $\OO{n}$-structures 
induced by this map on $\C{2} \setminus 0$ 
are invariant under linear isomorphisms of 
$\C{2}$. Therefore they quotient to every 
linear Hopf surface. We will refer to 
these $\OO{n}$-structures as the 
\emph{radial} $\OO{n}$-structures 
on Hopf surfaces. In affine coordinates, 
each radial structure has developing map
\[
 \left(z_1,z_2\right) 
\mapsto 
\left(t_1,t_2\right)
=\left(\frac{z_1}{z_2}, \frac{1}{z_2^n}\right)
\]
defined where $z_2 \ne 0$. Where $z_2=0$, 
we can just swap indices of $z_1$ and $z_2$ 
to get another affine chart, so we can see 
that the structure is holomorphic. The
image of the developing map is the
complement of the $0$-section in $\OO{n}$,
so the structure is incomplete. The 
developing map is an $n$-fold covering of 
its image, so is inessential.
The holonomy generator is $(g,p)=(F,0)$. More generally, 
swapping indices of $z_1$ and $z_2$ is an 
involution on the space of 
$\OO{n}$-structures on diagonal Hopf surfaces.


%

\subsection{The eigenstructures on linear Hopf surfaces}

A related example: consider $\C{2}$ 
foliated by vertical lines, i.e. the 
lines $z_1=\text{constant}$. The affine 
transformations of $\C{2}$ which preserve 
this foliation are precisely the maps of the form
\[
\begin{pmatrix}
z_1 \\
z_2
\end{pmatrix}
 \mapsto 
\begin{pmatrix}
a_{11} & 0 \\
a_{21} & a_{22}
\end{pmatrix}
\begin{pmatrix}
z_1 \\
z_2
\end{pmatrix}
+
\begin{pmatrix}
b_1 \\
b_2
\end{pmatrix}.
\]
with $a_{11} \ne 0 \text{ and } a_{22} \ne 0$.
Let $G_1$ be the group of all such
affine transformations, and $H_1$
the subgroup fixing the origin, i.e.
with $b_1=b_2=0$.
Clearly the graph of any  
polynomial function 
$z_2=z_2\left(z_1\right)$ of 
degree $n$
is carried to the graph of another
polynomial of the same degree
by any element of $G_1$. It follows 
(as we will shortly see) that there is 
an invariant $\OO{n}$-structure for 
which these graphs correspond to the 
global sections of $\OO{n} \to \Proj{1}$, 
and the vertical lines to the fibers.


Consider the subgroup $G_0 \subset G$ 
consisting of elements $(g,p) \in G$ of the form
\[
g = 
\begin{pmatrix}
g_{11} & g_{12} \\
0 & g_{22} 
\end{pmatrix},
\
p\left(Z_1,Z_2\right)
=
c_0 Z_2^{n} + c_1 Z_1 Z_2^{n-1}.
\]
Of course, $g$ is defined as a matrix
only up to scaling by $n$-th roots 
of $1$. Consider the complex Lie group isomorphism
\[
(g,p) \in G_0 \mapsto 
(a,b) \in G_1
\]
given by
\[
\begin{pmatrix}
a_{11} & 0 \\
a_{21} & a_{22}
\end{pmatrix}
=
\begin{pmatrix}
\frac{g_{11}}{g_{22}} & 0 \\
\frac{g_{11}}{g_{22}} c_1 & \frac{1}{g_{22}^n}
\end{pmatrix},
\
\begin{pmatrix}
b_1 \\
b_2
\end{pmatrix}
=
\begin{pmatrix}
\frac{g_{12}}{g_{22}} \\
c_0 + \frac{g_{11}}{g_{22}} c_1
\end{pmatrix}.
\]
This isomorphism identifies $G_1$ with the subgroup $G_0 \subset G$,
and $H_1$ with $H_0 \subset H$. Therefore a $G/H$-structure
is induced by a $G_1/H_1$-structure, i.e. an affine structure
foliated by parallel complex geodesics, if and only if
the holonomy of the $G/H$-structure lies in the subgroup 
$G_1$ and the developing map has image in $\OO{n}$ lying
inside the open orbit of $G_0$.

Another related example: suppose that $S$ 
is a complex surface with a complex affine 
structure, foliated by parallel geodesics. 
Locally we can construct coordinates 
$\left(t_1,t_2\right)$ on $S$ which 
identify open sets of $S$ with open sets 
of $\C{2}$, and identify the parallel 
geodesic foliation with the foliation 
of $\C{2}$ by vertical lines. Moreover, 
the transition maps will now preserve 
the vertical direction, and therefore 
are compositions of (1) translations, 
(2) rescalings of horizontal and vertical 
axes, and (3) addition of a linear 
function of $t_1$ to $t_2$. In particular, 
any graph of a polynomial function 
$t_2=t_2\left(t_1\right)$ will remain a 
graph of a polynomial of the same degree. 
In the standard flat affine structure 
on the torus, there is a foliation by 
parallel geodesics in each direction, 
and associated to each such foliation 
is a translation invariant 
$\OO{n}$-structure for every $n$.

Let's return now to Hopf surfaces. 
Pick an eigenline of a linear 
Poincar\'e domain map $F$. The affine lines 
parallel to that line form an 
$F$-invariant foliation of $\C{2}$. 
The associated  $\OO{n}$-structure 
descends to the associated Hopf surface. 
Let's call this the 
\emph{$\OO{n}$-eigenstructure}; 
there are two such for $F$ with two 
distinct eigenvalues, one for $F$ not 
diagonalizable, and infinitely many for 
$F$ a homothety. Up to isomorphism of 
the Hopf surface and perhaps a 
permutation of indices, the eigenline 
can be arranged to be the vertical axis, 
with contraction map of the form 
$F\left(z_1,z_2\right)
=\left(\lambda_1 z_1,\lambda_2 z_2\right)$. 
Then the developing map of the 
eigenstructure is given 
by $\text{id} : \left(z_1,z_2\right) 
\mapsto \left(t_1=z_1,t_2=z_2\right)$, 
and the holonomy generator is $(g,0)$ where 
\[
g=\begin{pmatrix}
\frac{\lambda_1}{\lambda_2^{1/n}} & 0 \\
0 & \frac{1}{\lambda_2^{1/n}}
  \end{pmatrix}.
\]
The image misses precisely
the fiber of $\OO{n} \to \Proj{1}$
over the point $\infty \in \Proj{1}$,
and the origin of the fiber of
$\OO{n} \to \Proj{1}$ over the point
$0 \in \Proj{1}$. 
The structure is incomplete but
is essential. 

Another example: take any compact curve 
$C$ of genus $g \ge 0$, and equip $C$ 
with a projective structure. (Every curve 
admits a projective structure; 
see \cite{Gunning:1978}.) 
The holonomy morphism
$\pi_1(C) \to \PSL{2,\C{}}$ lifts
to a morphism 
$\pi_1(C) \to \SL{2,\C{}}$
(see Gallo, Kapovich and Marden \cite{Gallo/Kapovich/Marden:2000}).
The developing map 
$\tilde{C} \to \Proj{1}$ of the projective
structure pulls back the 
$\OO{n}$-bundle to a line bundle over $C$. 
Cut out the zero section of this line bundle 
and quotient the fibers by any homothety 
$w \to a w$ with $a \ne 0$. The result is 
an $\OO{n}$-structure on a principal 
fibration by elliptic curves over the 
curve of genus $g$. The developing maps 
of projective structures 
of curves of large genus can be very 
complicated (see 
Gallo, Kapovitch and Marden \cite{Gallo/Kapovich/Marden:2000}), 
so the developing map of the 
$\OO{n}$-structure cannot be made explicit.
The holonomy morphism 
$\pi_1(C) \rtimes \Z{} \to G$
takes the generator of $\Z{}$ to
the homothety, and $\pi_1(C) \to \GL{2,\C{}}/(\text{$n$-th roots of 1})$
is the lift of the holonomy morphism
of the projective structure on $C$. 
The structure is incomplete, but
is essential.

\subsection{The eigenstructures 
on exceptional Hopf surfaces}

Pick any integers $0 < m \le n$. 
Let $(g,p) \in G$ be the element
\[
 g=
\begin{pmatrix}
\frac{\lambda}{\varepsilon} & 0 \\
0 & \frac{1}{\varepsilon}
\end{pmatrix},
\
p\left(Z_1,Z_2\right)=\frac{1}{\lambda^m} Z_1^m Z_2^{n-m}
\]
where $\varepsilon$ is 
any solution of $\varepsilon^n=\lambda^m$.
(We obtain the same element $(g,p) \in G$
for any choice of $\varepsilon$.)
This element $(g,p)$ acts on $\OO{n}$ via
\[
 (g,p)\left(t_1,t_2\right)
=\left(\lambda t_1, \lambda^m t_2 + t_1^m\right).
\]
In particular, every exceptional Hopf 
surface of degree $m$ has an 
$\OO{n}$-structure, for all $n \ge m$, 
which we call the \emph{eigenstructure} 
on the exceptional Hopf surface. The 
holonomy generator is $(g,p)$, and the 
developing map is 
$\text{id} : \left(z_1,z_2\right)\mapsto\left(t_1=z_1,t_2=z_2\right)$.
The image misses precisely
the fiber of $\OO{n} \to \Proj{1}$
over the point $\infty \in \Proj{1}$,
and the origin of the fiber of
$\OO{n} \to \Proj{1}$ over the point
$0 \in \Proj{1}$. In particular,
the structure is incomplete, but
is essential.

\subsection{The hyperresonant structures on hyperresonant Hopf surfaces}
\begin{definition}
A hyperresonant Hopf surface
with hyperresonance $\lambda_1^{m_1}=\lambda_2^{m_2}$
may have additional $\OO{n}$-structures,
which we will call \emph{hyperresonant} 
structures. It has such structures just when it satisfies the conditions given
in table~\vref{table:exoticStructures},
as we will see in section~\vref{subsubsection:GenericHolonomyOnDiagonal}.
\begin{table}
\begin{displaymath}
\begin{array}{lll}
\text{condition}& 
\text{developing map} & 
\text{holonomy generator} \\ \midrule
m_2 = n \, m_1, m_1 \ge 2 \text{ or } n \ge 2 & 
 \left(\frac{z_1^{m_1}-a_1 \, z_2^{m_2}}{z_2},\frac{z_1}{z_2^n}\right) & 
 \begin{pmatrix}
  \frac{\lambda_1^{m_1}}{\lambda_1^{1/n}} & 0 \\
 0 & \frac{\lambda_2}{\lambda_1^{1/n}}
 \end{pmatrix}
  \\ 
m_2 = n \, m_1, N \ge 2 &
  \left(\frac{\prod_{j=1}^{N} \left(z_1^{m_1}-a_j \, z_2^{m_2}\right)}%
  {z_2},
  \frac{z_1}{z_2^n}
\right) & 
\begin{pmatrix}
 \frac{\lambda_1^{m_1 N}}{\lambda_1^{1/n}} & 0 \\
0 & \frac{\lambda_2}{\lambda_1^{1/n}}
\end{pmatrix}
 \\ 
m_1=m_2, m_1 N \ne n &
\left(
  \frac{z_1}{z_2} , 
  \frac{\prod_{j=1}^N 
  \left( 
    z_1^{m_1} - a_j \, z_2^{m_2}
  \right)}{z_2^n}
\right)
&
\begin{pmatrix}
\frac{\lambda_1}{\lambda_2^{m_2 N/n}} & 0 \\
0 & \frac{\lambda_2}{\lambda_2^{m_2 N/n}}
\end{pmatrix}
 \\ 
m_1 = n \, m_2, m_2 \ge 2 \text{ or } n \ge 2 &
\left(\frac{z_1}{z_1^{m_1}-a_1 \, z_2^{m_2}}, 
\frac{z_2}{\left(z_1^{m_1}-a_1 \, z_2^{m_2}\right)^n} \right) & 
\begin{pmatrix}
\frac{\lambda_1}{\lambda_2^{1/n}} & 0 \\
0 & \frac{\lambda_1^{m_1}}{\lambda_2^{1/n}}
\end{pmatrix}
\\
m_1 = n \, m_2, N \ge 2 &
\left(\frac{z_1}{\prod_{j=1}^N\left(z_1^{m_1}-a_j \, z_2^{m_2}\right)}, 
\frac{z_2}{\prod_{j=1}^N\left(z_1^{m_1}-a_j \, z_2^{m_2}\right)^n} \right) & 
\begin{pmatrix}
\frac{\lambda_1}{\lambda_2^{1/n}} & 0 \\
0 & \frac{\lambda_1^{m_1 N}}{\lambda_2^{1/n}}
\end{pmatrix}
\end{array}
\end{displaymath}
\caption{The hyperresonant structures on hyperresonant Hopf surfaces.
The quantities $a_1, a_2, \dots, a_N$ are 
any distinct nonzero complex constants.
The holonomy generator in each case is actually $(g,0)$
where $g$ is the matrix given in the last column above.
When a matrix contains an $n$-th root, like $\lambda_1^{1/n}$,
the same value of
the $n$-th root must be used in every entry in that matrix.}\label{table:exoticStructures}
\end{table}
\end{definition}
None of these are complete or essential structures, as the reader can easily check. The
images in $\OO{n}$ of the developing maps are complicated.
The developing maps cover their images as finite unramified covering maps, with
more than one sheet.

\section{Classification on Hopf surfaces}\label{section:Classification}

\begin{theorem}\label{theorem:Classification}
Up to isomorphism, the $\OO{n}$-structures on Hopf surfaces are
precisely those given in table~\vref{table:classification},
i.e. precisely the examples given in 
tables~\vref{table:examples} and \vref{table:exoticStructures}.
\end{theorem}
The proof of this theorem will occupy the remainder
of this section. Note that every Hopf surface admits an 
$\OO{n}$-structure for some value of $n$. 
Every linear Hopf surface admits an $\OO{n}$-structure
for all $n \ge 1$. Every nonlinear
Hopf surface only admits $\OO{n}$-structures for
$n \ge m$ where $m$ is the degree of the Hopf surface.

Roughly speaking, even among hyperresonant
Hopf surfaces, hyperresonant structures are somewhat
rare. To be precise: a hyperresonant Hopf 
surface admits a hyperresonant
$\OO{n}$-structure if and only if its hyperresonance
$\left(m_1,m_2\right)$ has $m_2 = n \, m_1$ for some
integer $n$, and then either (1) only admits hyperresonant 
$\OO{n}$-structures 
for that integer $n$ or (2) if $m_1=m_2$, 
admits hyperresonant $\OO{n}$-structures 
for all $n \ge 1$. 

\begin{table}
\begin{center}
\begin{tabular}{rcccc}
                      & \multicolumn{3}{c}{structure} \\
                      \cmidrule(l){2-4} 
surface               & radial          & eigenstructure  & hyperresonant 
\\ \midrule
generic               & \yes            & \yes            & \no 
\\
hyperresonant         & \yes            & \yes            & \yes 
\\
exceptional linear    & \yes            & \yes            & \no 
\\
exceptional nonlinear & \no             & \yes            & \no 
\end{tabular}
\end{center}
\caption{The classification}\label{table:classification}
\end{table}

\begin{table}
\begin{center}
\begin{tabular}{rcc}
structure         & complete & essential
\\ \midrule
radial               & \no      & \no \\
eigenstructure       & \no      & \yes \\
hyperresonant        & \no      & \no 
\end{tabular}
\end{center}
\caption{Completeness and essentiality}\label{table:completeness}
\end{table}

\subsection{Diagonal Hopf surfaces}
Let's suppose that $S_F$ is a Hopf surface 
and $F$ is diagonal linear, say
\[
 F(z)=\left(\lambda_1 z_1, \lambda_2 z_2\right).
\]
Each $\OO{n}$-structure on $S_F$ has 
developing map a local biholomorphism 
$\dev : \C{2} \setminus 0 \to \OO{n}$. 
There is a holonomy generator 
$\hol=(g,p) \in G$ so that 
$\dev(F(z))=(g,p)\dev(z)$. In the 
coordinates $\left(t_1,t_2\right)$ 
we constructed above on $\OO{n}$, 
the developing map is a pair of
meromorphic functions 
$\left(t_1,t_2\right)=
\left(
  t_1\left(z_1,z_2\right),
  t_2\left(z_1,z_2\right)
\right)$. In particular, 
$t_1\left(\lambda_1 z_1, \lambda_2 z_2\right)
=
g \, t_1\left(z_1,z_2\right)$, 
where $g$ acts here by linear fractional 
transformation. So $t_1$ is a meromorphic 
section of a $\Proj{1}$-bundle over our 
Hopf surface. By proposition~\vref{proposition:MeromorphicSections}, 
$t_1$ must have the form
\[
 t_1\left(z_1,z_2\right)
=z_1^{k_1} z_2^{k_2} 
\frac{P_1(u)}{Q_1(u)} 
\ \text{ where } u=z_1^{m_1}/z_2^{m_2},
\]
and either (1) $P_1$ and $Q_1$ are constants or 
(2) $\lambda_1^{m_1}=\lambda_2^{m_2}$ is a 
hyperresonance. At the expense of changing the values
of $k_1$ and $k_2$, we can assume that $P_1(u)$ 
and $Q_1(u)$ each have no zeroes at $u=0$, 
and that they have no zeroes in common. The 
map $t_1$ is the composition 
$\C{2} \setminus 0 \to \OO{n} \to \Proj{1}$, 
a composition of holomorphic submersions, 
so a holomorphic submersion. Therefore 
$k_1=-1,0$ or $1$, and $P_1$ and $Q_1$ can't 
have multiple zeroes. The map $t_1$ as 
written is not defined at $z_2=0$, and we 
have to rewrite it in order to examine 
its behaviour near $z_2=0$. It is 
convenient to rewrite the map as
\begin{align*}
 t_1\left(z_1,z_2\right)
&=
z_1^{\tilde{k}_1} 
z_2^{\tilde{k}_2} 
\frac
{%
  \tilde{P}_1
  \left(
    \tilde{u}
  \right)
}%
{%
  \tilde{Q}_1
  \left(
    \tilde{u}
  \right)
}%
\ 
\text{ where } \\
\tilde{u}
&=
\frac{1}{u}
=z_2^{m_2}/z_1^{m_1}, \\
\tilde{k}_1
&=
k_1+m_1\left(\deg P_1 - \deg Q_1\right) \\
\tilde{k}_2
&=
k_2-m_2\left(\deg P_1 - \deg Q_1\right)
\end{align*}
and the roots of $\tilde{P}_1$ and $\tilde{Q}_1$ 
are the reciprocals of those of $P_1$ and $Q_1$. 
Then $\tilde{k}_2$ must also be among 
$-1, 0$, or $1$. These conditions together 
ensure that 
$t_1 : \C{2} \setminus 0 \to \Proj{1}$ 
is a holomorphic submersion, and 
satisfies $t_1(F(z))=g \, t_1(z)$. 
Moreover, they force $g$ to have the form
\[
 g=\begin{pmatrix}
    c \lambda_1^{k_1} \lambda_2^{k_2} & 0 \\
0 & c
   \end{pmatrix}
\]
with $k_1=-1,0$ or $1$ and 
$\tilde{k}_2=-1, 0$ or $1$ and $c \ne 0$.

\subsubsection{Generic holonomy on diagonal Hopf surfaces}\label{subsubsection:GenericHolonomyOnDiagonal}
Consider a diagonal Hopf surface with $\OO{n}$-structure.
Assume that the holonomy generator 
$(g,p)$ is generic, i.e. has the 
form $(g,0)$, and the surface is diagonal.
\begin{definition}
A \emph{semiadmissible map}
for a diagonal Hopf surface $S_F$ with
\[
F = 
\begin{pmatrix}
\lambda_1 & 0 \\
0 & \lambda_2
\end{pmatrix}
\]
and with hyperresonance $\lambda_1^{m_1}=\lambda_2^{m_2}$
is a map $t : \C{2} \setminus 0 \to \OO{n}$ of the form
\begin{align}\label{eqn:ts}
 \left(t_1,t_2\right)
&=
\left(
z_1^{k_1} z_2^{k_2} 
\frac{P_1(u)}{Q_1(u)}, 
z_1^{\ell_1} z_2^{\ell_2} 
\frac{P_2(u)}{Q_1(u)^n}
\right)
\\
 \left(s_1,s_2\right)
&=
\left(
z_1^{-k_1} z_2^{-k_2} 
\frac{Q_1(u)}{P_1(u)}, 
z_1^{\ell_1-n k_1} 
z_2^{\ell_2- n k_2} 
\frac{P_2(u)}{P_1(u)^n}
\right),
\end{align}
so that 
\begin{enumerate}
\item the expressions $P_1(u), Q_1(u), P_2(u)$ 
are polynomials, where 
$u=z_1^{m_1}/z_2^{m_2}$ and
\item none of these polynomials have any 
double roots, or roots at $u=0$, and 
\item
no two of them have any common roots and
\item
$k_1$ and $\ell_1$ belong to the following list:
\begin{displaymath}\label{table:kl}
\begin{array}{cc}
k_1 & \ell_1 \\ 
\midrule 
\addlinespace[2pt]
-1 & -n \\ 
0 & 0 \\
0 & 1 \\ 
1 & 0 \\
\end{array}
\end{displaymath}
and
\item
the numbers $\tilde{k}_2, \tilde{\ell}_2$ 
belong to this same list, where 
$\tilde{k}_2=k_2-m_2 \left(\deg P_1 - \deg Q_1\right)$ 
and 
$\tilde{\ell}_2
=\ell_2-m_2\left(\deg P_1 - n \, \deg Q_1\right)$. 
\end{enumerate}
(We will discuss semiadmissible maps in this section only.)
\end{definition}

\begin{lemma}\label{lemma:genericHolonomy}
 Suppose that $F : \C{2} \to \C{2}$ 
is a diagonal linear map in the 
Poincar\'e domain. 
A map $t : \C{2} \setminus 0 \to \OO{n}$ is semiadmissible
for the Hopf surface $S_F$
if and only if it is the developing map of a branched
$\OO{n}$-structure on $S_F$. The holonomy of the branched
$\OO{n}$-structure is $(g,0)$ where
\[
 g = 
\begin{pmatrix}
\lambda_1^{k_1+\ell_1} \lambda_2^{k_2+\ell_2} & 0 \\
0 & \lambda_1^{\ell_1} \lambda_2^{\ell_2}
\end{pmatrix}.
\]
\end{lemma}
\begin{proof}
A semiadmissible map is equivariant under the action of $F$, so 
provides a branched $\OO{n}$-structure. 

Clearly 
$t_2\left(\lambda_1 z_1, \lambda_2 z_2\right)
=\frac{1}{c^n}t_2\left(z_1,z_2\right)$. 
So $t_2$ is also a meromorphic section of 
a $\Proj{1}$-bundle on the Hopf surface. 
Therefore $c^n=\lambda_1^{\ell_1} \lambda_2^{\ell_2}$ 
for some integers $\ell_1$ and $\ell_2$, and 
\[
  t_2\left(z_1,z_2\right)
=z_1^{\ell_1} z_2^{\ell_2} \frac{P_2(u)}{Q_2(u)} 
\ \text{ where } u=z_1^{m_1}/z_2^{m_2}.
\]
Again, we can assume that $P_2$ and $Q_2$ 
have no roots at $0$ and no roots in common. 
To have $\left(t_1, t_2\right)$ a local 
biholomorphism, $P_2$ can't have any double 
roots, and $\ell_1, \ell_2 \le 1$. One of the two mappings
\begin{align*}
 \left(t_1,t_2\right)
&=
\left(
z_1^{k_1} z_2^{k_2} \frac{P_1(u)}{Q_1(u)}, 
z_1^{\ell_1} z_2^{\ell_2} \frac{P_2(u)}{Q_2(u)}
\right)
\\
 \left(s_1,s_2\right)
&=
\left(
z_1^{-k_1} z_2^{-k_2} \frac{Q_1(u)}{P_1(u)}, 
z_1^{\ell_1-n k_1} z_2^{\ell_2- n k_2} 
\frac{P_2(u) Q_1(u)^n}{Q_2(u) P_1(u)^n}
\right)
\end{align*}
must be defined at each point 
$\left(z_1, z_2\right) \in \C{2} \setminus 0$. 
It is not possible for $t_1$ and $t_2$ 
to have a common polynomial factor, 
because where this polynomial vanishes, 
the developing map will not be a local 
biholomorphism. When we look at the 
line $z_1=0$, we see that this constrains 
us to the table of $\left(k_1,\ell_1\right)$ 
values above, and exactly 
the same is true of 
$\left(\tilde{k}_2,\tilde{\ell}_2\right)$ 
by the same reasoning.
Clearly $P_2$ cannot have any double 
zeroes, or else $\left(t_1,t_2\right)$ won't 
be a local biholomorphism. The functions 
$t_1, t_2$ cannot be defined where 
$Q_1=0$ or where $Q_2=0$, so 
$\left(s_1,s_2\right)$ must be defined. 
Therefore each zero of $Q_2$ must cancel 
a zero of $Q_1^n$. But then the leftover 
zeroes of $Q_1^n$ cannot be double zeroes, 
so each zero of $Q_2$ must occur with 
multiplicity precisely $n-1$ or $n$. 
We can write $Q_2 = R Q_1^{n-1}$, with $R$ 
dividing $Q_1$ and having no double roots. 
So we can write $Q_1 = RS$, with $R$ and $S$ 
having no common roots. It then follows that 
$S$ is a factor of $s_1$ and $s_2$, so 
$s_1$ and $s_2$ have linearly dependent 
differentials at points where $S=0$. 
These are points at which both $s_1$ and 
$s_2$ are holomorphic. Therefore $S$ must be 
constant. Absorbing the constant, we have $Q_2=Q_1^n$.
\end{proof}

Calculation on any semiadmissible map yields 
\begin{align*}
\pd{t_1}{z_1} 
&= 
\frac{t_1}{z_1}\left( k_1 + m_1 u \left(\frac{P_1'}{P_1} - \frac{Q_1'}{Q_1}\right) \right), 
\\
\pd{t_1}{z_2} 
&= 
\frac{t_1}{z_2}\left( k_2 - m_2 u \left(\frac{P_1'}{P_1} - \frac{Q_1'}{Q_1}\right) \right), 
\\
\pd{t_2}{z_1} 
&= 
\frac{t_2}{z_1}\left( \ell_1 + m_1 u \left(\frac{P_2'}{P_2} - n \frac{Q_1'}{Q_1}\right) \right), 
\\
\pd{t_2}{z_2} 
&= 
\frac{t_2}{z_2}\left( \ell_2 - m_2 u \left(\frac{P_2'}{P_2} - n \frac{Q_1'}{Q_1}\right) \right).
\end{align*}
and so
\[
\det t' = z_1^{k_1+\ell_1-1} 
z_2^{k_2 + l_2 -1} 
\frac{P_1(u) P_2(u)}{Q_1(u)^{n+1}} R(u)
\]
where
\[
R(u)= A \, \frac{u \, P_1'(u)}{P_1(u)} 
- B \, \frac{u \, P_2'(u)}{P_2(u)} 
+ C \, \frac{u \, Q_1'(u)}{Q_1(u)} + D,
\]
and
\begin{align*}
A &= m_1 \ell_2 + \ell_1 m_2 \\
B &= m_1 k_2 + k_1 m_2 \\
C &= n\left(m_1 k_2 + k_1 m_2 \right) 
- \left(m_1 \ell_2 + \ell_1 m_2 \right) \\
D &=  k_1 \ell_2 - \ell_1 k_2.
\end{align*}
We can see that semiadmissible maps
are local biholomorphisms near $z_1=0$ and 
near $z_2=0$. They determine branched 
$\OO{n}$-structures. However, in order 
that the branched structure of a semiadmissible map be an 
unbranched structure, $\det t'$ must have 
no zeroes except at points where 
$\left(t_1, t_2\right)$ are not defined. Recall that if 
\[
 P_1(u) = c \, \prod_j \left(u-a_j\right)
\]
then
\[
 \frac{u P_1'(u)}{P_1(u)} = u \sum_j \frac{1}{u-a_j}. 
\]
The function $R(u)$ has simple poles 
at the zeroes of $P_1(u)$ if $A \ne 0$, 
at the zeroes of $P_2(u)$ if $B \ne 0$, 
and at the zeroes of $Q_1(u)$ if $C \ne 0$. 
If $R(u)$ has any zeroes at finite 
values of $u$, then in order to keep 
$\det t' \ne 0$, we would need to have 
those zeroes occur somewhere where they 
can cancel out with poles from 
$z_1^{k_1+\ell_1-1} z_2^{k_2 + l_2 -1} 
\frac{P_1(u) P_2(u)}{Q_1(u)^{n+1}}$. 
So the finite zeroes of $R(u)$ only occur 
at zeroes of $Q_1(u)$. 

If $C \ne 0$, then $R(u)$ has poles at all
of the zeroes of $Q_1(u)$, so no 
cancellations take place. But then 
$R(u)$ must have at least $\deg Q_1$ zeroes 
(counting with multiplicity), so these must 
lie at infinity. If $Q_1(u)$ is constant, 
then no cancellations can take place, so 
$R(u)$ can't have any zeroes at finite values 
of $u$, so again all zeroes of $R(u)$ are at 
infinity. Therefore $Q_1(u)$ is constant 
or $C=0$ or all zeroes of $R(u)$ are at infinity.

\begin{definition}
 A semiadmissible map is \emph{admissible}
if (in the above notation from this section)
\begin{enumerate}
\item either $A=0$ or $P_1(u)$ is constant and
\item either $B=0$ or $P_2(u)$ is constant and 
\item either $C=0$ or $Q_1(u)$ is constant and
\item $R(u)=D$ is constant, not zero and
\item $k_1 \ne 0$ or $\ell_1 \ne 0$.
\end{enumerate}
(We will discuss admissible maps in this section only.)
\end{definition}

\begin{lemma}
A semiadmissible map which is local biholomorphism at every point 
of $\C{2} \setminus 0$ (i.e. not branched) is admissible.
\end{lemma}
\begin{proof}
Suppose that $A \ne 0$. Let's pick 
one of the zeroes of $P_1(u)$, 
say $a_1$, and write $R(u)$ as
\[
 R(u) = \frac{Au}{u-a_1} + f(u).
\]
Now solve $R(u)=0$ for $a_1$:
\[
 a_1 = u \left( 1 + \frac{A}{f} \right).
\]
Imagine varying the choice of $P_1$, by varying $a_1$ and 
leaving the other linear factors of $P_1$
intact. We thereby vary the choice of $t$ and so of $R$.
We see that in order to force $R(u)=0$ at a given value of $u$, 
we only have to set $a_1$ as above. For generic choice of $u$, there is 
therefore a unique choice of $a_1$ which will ensure $R(u)=0$. So if we 
$P_1(u)$ is not constant, we can slightly alter $P_1(u)$ to ensure 
that $R(u)$ has all its zeroes at finite locations away from the zeroes of $Q_1(u)$. 
Therefore generic choice of $P_1(u)$ will lead to a branched $\OO{n}$-structure, 
which has a nontrivial branch locus. The limit of the branch locus is 
still a compact curve in the Hopf surface, since the space of curves is compact.
Therefore if $A \ne 0$ and $P_1(u)$ is not constant, then the 
branched structure has nonempty branch locus. The same proof works 
for $P_2(u)$.

By the same argument, if $Q_1(u)$ is not constant, and $C \ne 0$, 
then we can perturb to a branched structure, which has 
finite roots for $R(u)$. Therefore this perturbed 
structure must have nonempty branch locus, and so our 
original branched structure had nonempty branch locus.
\end{proof}

\begin{lemma}
A semiadmissible map is admissible if and only if its
branch locus is empty, i.e. it is the developing map
of an $\OO{n}$-structure.
\end{lemma}
\begin{proof}
If we have an admissible map then our branched
structure has no branch locus in the region in which the $\left(t_1, t_2\right)$ 
functions in equation~\ref{eqn:ts} are defined. Therefore we only 
need to then check the branch locus in all 
four coordinate charts: the $\left(t_1, t_2\right)$ and 
$\left(s_1,s_2\right)$ charts on $\OO{n}$, and the expressions in 
$z_1,z_2$ with and without $\tilde{}$ symbols on them, the charts on the Hopf surface. 
It is easy to check that when we change
to the $\tilde{}$ symbol coordinates, the
corresponding quantities, in the obvious notation, are
\begin{align*}
\tilde{u}
&=
\frac{1}{u}
=z_2^{m_2}/z_1^{m_1}, \\
\tilde{k}_1
&=
k_1+m_1\left(\deg P_1 - \deg Q_1\right), \\
\tilde{k}_2
&=
k_2-m_2\left(\deg P_1 - \deg Q_1\right), \\
 \tilde{m}_1 &= -m_1, \\
 \tilde{m}_2 &= -m_2, \\
 \tilde{\ell}_1 &= \ell_1 + m_1 \left( \deg P_2 - n \, \deg Q_1 \right), \\
 \tilde{\ell}_2 &= \ell_2 - m_2 \left( \deg P_2 - n \, \deg Q_1 \right), \\
\tilde{A} &= -A, \\
\tilde{B} &= -B, \\
\tilde{C} &= -C, \\
\tilde{D} &= D + A \left( \deg P_1 - \deg Q_1 \right ) 
- B \left ( \deg P_2 - n \, \deg Q_1 \right).
\end{align*}
We can easily see that $\tilde{P}_1\left(\tilde{u}\right)$
is constant just when $P_1(u)$ is constant and $\tilde{A}=0$
just when $A=0$, etc. Therefore admissibility is unchanged
by such a coordinate transformation.

Let's write out our map in $\left(s_1,s_2\right)$ coordinates,
say
\begin{align*}
\left(s_1,s_2\right)
&=
\left(
z_1^{\hat{k}_1} z_2^{\hat{k}_2} 
\frac{\hat{P}_1(u)}{\hat{Q}_1(u)}, 
z_1^{\hat{\ell}_1} z_2^{\hat{\ell}_2} 
\frac{\hat{P}_2(u)}{\hat{Q}_1(u)^n}
\right)
\\
&=
\left(
z_1^{-k_1} z_2^{-k_2} 
\frac{Q_1(u)}{P_1(u)}, 
z_1^{\ell_1-n k_1} 
z_2^{\ell_2- n k_2} 
\frac{P_2(u)}{P_1(u)^n}
\right).
\end{align*}
Then we find the dictionary
\begin{align*}
\hat{k}_1 &= -k_1, \\
\hat{k}_2 &= -k_2, \\
\hat{m}_1 &= m_1, \\
\hat{m}_2 &= m_2, \\
\hat{l}_1 &= \ell_1 - n \, k_1, \\
\hat{l}_2 &= \ell_2 - n \, k_2, \\
\hat{A} &= A - n \, B, \\
\hat{B} &= -B, \\
\hat{C} &= -A, \\
\hat{D} &= -D.
\end{align*}
Again, the admissibility of a semiadmissible map 
is unchanged by this coordinate transformation.
\end{proof}

By semiadmissibility, we need $D \ne 0$, so $\left(k_1,\ell_1\right) \ne \left(0,0\right)$.
By admissibility, we will also need $\tilde{D} \ne 0$ so
$\left(\tilde{k}_2,\tilde{\ell}_2\right) \ne (0,0)$. We now have a
tedious computation: for each of the 3 possible values of
$\left(k_1,\ell_1\right)$ and the 3 possible values of $\left(\tilde{k}_2,\tilde{\ell}_2\right)$
from table~\vref{table:kl} (except for $(0,0)$), we calculate $A,B,C,D, \tilde{A}, \tilde{B}, \tilde{C}, \tilde{D}$.
In order that $P_1, Q_1$ and $P_2$ not be all constant, we will need one of $\deg P_1, \deg Q_1, \deg P_2$ to
be nonzero. If $P_1$ is not constant, then $A=0$, etc. To keep the branched $\OO{n}$-structure from branching, 
we will need $D \ne 0, \tilde{D} \ne 0,$ and $\hat{D} \ne 0$. The full story of manipulating
these inequalities is totally elementary, so we will only explain fully one
of the nine cases, and then leave all others to the reader, to avoid many pages of
elementary arguments with inequalities. 

Consider the case of an admissible map with $k_1=0, \ell_1 = 1, \tilde{k}_2=-1, \tilde{\ell}_2=-n$.
We find
\begin{align*}
 A &= m_1 n - m_2 - m_1 m_2 \left( \deg P_2 - n \deg Q_1 \right), \\
 B &= m_1 \left( -1 + m_2 \left( \deg P_1 - \deg Q_1 \right) \right), \\
 C &= m_2 \left( m_1 n \deg P_1 - m_1 \deg P_2 -1 \right), \\
 D &= -1 + m_2 \left(\deg P_1 - \deg Q_1 \right).
\end{align*}
In particular, $B= m_1 D$ and $m_1 \ne 0$ because the hyperresonance
has $m_1,m_2 \ge 1$. Moreover $D \ne 0$, since the structure is not branched.
Therefore $B \ne 0$. By admissibility, $P_2$ is a constant.
We compute that $\hat{D} = -1-m_1 \deg P_2 + m_1 n \deg P_1$. Therefore
$C = m_2 \hat{D}$. We can therefore say that $C \ne 0$ and so
$Q_1$ is constant by admissibility. Assume that not all of $P_1, P_2, Q_1$ are constant. 
Clearly now $A=m_1 \, n - m_2$, and so $m_2 = m_1 n$. Plugging this in gives $0 \ne D=-1 + m_1 n \deg P_1$, so
$m_1 n \deg P_1 \ne 1$, and so $m_1 n \deg P_1 > 1$ since $m_1, n, \deg P_1 \ge 1$.
Therefore $m_1 > 1$ or $n > 1$ or $\deg P_1 > 1$.
All eight other cases follow essentially the same reasoning.

Tedious computation of all nine cases yields the conditions of 
table~\vref{table:mapleCases} in order that the structure is unbranched and $P_1(u), P_2(u)$ 
and $Q_1(u)$ are not all forced to be constant. The
impossible cases come from inconsistency of the degrees
of the polynomials $P_1, Q_1$, and $P_2$.

\begin{table}
{\tiny{%
\begin{displaymath}%
\begin{array}{lllllllll}
k_1 & \ell_1 & \tilde{k}_2 & \tilde{l}_2 &                  \deg P_1 &                  \deg Q_1 &                               \deg P_2 &       \text{and\ldots}  
\\ \midrule
0   &      1 &          -1 &          -n &                 {} \ge 1  &                         0 &                                      0 &             m_2=n \, m_1, \\
    &        &             &             &                           &                           &                                        & m_1>1 \text{ or } n > 1 \text{ or } \deg P_1>1 \\ 
\\
0   &      1 &           0 &           1 & \frac{m_1+m_2}{m_1 m_2 n} & \frac{m_1+m_2}{m_1 m_2 n} &                                      0 &   \deg P_1 \ne \deg Q_1, \\ 
& & & & & & & \text{impossible} \\ 
\\
0   &      1 &           1 &           0 &                         0 &                  {} \ge 1 &                                      0 & m_2 = n \, m_1, \\
& & & & & & & m_1>1 \text{ or } n > 1 \text{ or } \deg Q_1>1 \\ 
\\
1   &      0 &          -1 &          -n &                         0 &                         0 &                               {} \ge 1 & m_1=m_2, \\
& & & & & & & m_1 \deg P_2 \ne n \\ 
\\
1   &      0 &           0 &           1 &                         0 &                  {} \ge 1 &                                      0 & m_1 = n \, m_2, \\
& & & & & & & m_2 > 1 \text{ or } n > 1 \text{ or } \deg Q_1>1 \\ \\
1   &      0 &           1 &           0 &                         0 &   \frac{m_1+m_2}{m_1 m_2} & \frac{n \left(m_1+m_2\right)}{m_1 m_2} & \deg P_2 \ne n \, \deg Q_1, \\
& & & & & & & \text{impossible}
\end{array}%
\end{displaymath}%
}}
\caption{The 6 cases of developing maps with at least one of $P_1, Q_1, P_2$ not constant}\label{table:mapleCases}
\end{table}

On the other hand, if $P_1, Q_1$ and $P_2$ are all assumed to be constants, and
the structure is unbranched (i.e. $D \ne 0$) then 
we can arrange that $k_1 \ge 0$ by replacing $\left(t_1,t_2\right)$
coordinates by $\left(s_1,s_2\right)$ coordinates, and
then we see that the possible values of $k_1, \ell_1, k_2$ and $\ell_2$ are
given in table~\vref{table:ConstantCasesII}.
\begin{table}
\begin{displaymath}
\begin{array}{lllllllll}
k_1 & \ell_1 & \tilde{k}_2 & \tilde{l}_2 & t_1 & t_2 \\ \midrule
0 & 1 & -1 & -n & \frac{1}{z_2} & \frac{z_1}{z_2^n} \\ 
0 & 1 & 1 & 0 & z_2 & z_1 \\ 
1 & 0 & -1 & -n & \frac{z_1}{z_2} & \frac{1}{z_2^n} \\ 
1 & 0 & 0 & 1 & z_1 & z_2
\end{array}
\end{displaymath}
\caption{The possible values of the integers $k_1, \ell_1, k_2, \ell_2$ 
and associated developing maps if $P_1, Q_1$ and $P_2$ are constant}\label{table:ConstantCasesII}
\end{table}


Note that with $P_1, Q_1$ and $P_2$ constants, we can
rescale $z_1$ and $z_2$ independently, since these
rescalings commute with our linear map $F$, and
thereby absorb constants as needed. After such absorptions,
we find the developing maps in table~\vref{table:ConstantCasesII}.
It is easy to see that the first line of this table
is isomorphic to the second via the isomorphism
$(g,p)=(g,0)$ with
\[
g=
\begin{pmatrix}
 0 & 1 \\
 1 & 0
\end{pmatrix}.
\]
Clearly the second and fourth lines of table~\vref{table:ConstantCasesII}
are eigenstructures, while the third is a radial structure.

\begin{proposition}
On any generic Hopf surface, the $\OO{n}$-structures with generic holonomy, up to 
isomorphism, are the radial structures and eigenstructures.
\end{proposition}
\begin{proof}
Just plug in the values 
in table~\vref{table:ConstantCasesII} into the general expression
of a semiadmissible map, with the added information that all of the polynomials in $u$ must be constants, to find that
up to isomorphism:
\begin{displaymath}
\begin{array}{ll}
\text{developing map} & \text{holonomy generator} \\ \midrule
\left(\frac{1}{z_2},\frac{z_1}{z_2^n}\right) & 
\begin{pmatrix}
 \frac{1}{\lambda_1^{1/n}} & 0 \\
0 & \frac{\lambda_2}{\lambda_1^{1/n}}
\end{pmatrix}
 \\ 
\left(z_2, z_1\right) & 
\begin{pmatrix}
 \frac{\lambda_2}{\lambda_1^{1/n}} & 0 \\
0 & \frac{1}{\lambda_1^{1/n}}
\end{pmatrix}
 \\ 
\left(\frac{z_1}{z_2} , \frac{1}{z_2^n}\right)
&
\begin{pmatrix}
 \lambda_1 & 0 \\
0 & \lambda_2
\end{pmatrix}
 \\ 
\left(z_1, z_2\right) & 
\begin{pmatrix}
 \frac{\lambda_1}{\lambda_2^{1/n}} & 0 \\
0 & \frac{1}{\lambda_2^{1/n}}
\end{pmatrix}
\end{array}
\end{displaymath}
\end{proof}
\begin{proposition}
On any hyperresonant Hopf surface, with hyperresonance $\lambda_1^{m_1}=\lambda_2^{m_2}$, 
the $\OO{n}$-structures with generic holonomy, up to isomorphism, are precisely
\begin{enumerate}
\item the radial structures, 
\item
the eigenstructures, and
\item 
the hyperresonant structures.
\end{enumerate}
\end{proposition}
\begin{proof}
Just plug in the values from table~\vref{table:mapleCases} and you find
table~\vref{table:exoticStructures}.
\end{proof}

\subsubsection{Nongeneric holonomy on diagonal Hopf surfaces}

\begin{lemma}\label{lemma:InfiniteOrder}
 Suppose that $\left(g,p\right)$ is the holonomy
of an $\OO{n}$-structure on a Hopf surface.
Then $g$ has infinite order, i.e. $g^N \ne I$
for any integer $N \ne 0$.
\end{lemma}
\begin{proof}
Suppose that $g$ has finite order, say $g^N=I$, $N \ge 1$.
By lemma~\vref{lemma:holonomyNormalForm}, 
we can assume that $(g,p)$ is in normal form.
In particular we can assume that 
$g$ is diagonal. Suppose that $\left(g,p\right)$
is the holonomy of an $\OO{n}$-structure on a Hopf
surface $S_F$. 
Let $F^N$ be the $N$-fold composition $F \circ F \circ \dots F$. 
Then $\left(g^N,Np\right)$ is the holonomy
of the pullback $\OO{n}$-structure on the Hopf surface
$S_{F^N}$ via the obvious covering map $S_{F^N} \to S_F$.
Therefore, by possibly replacing $S_F$ with $S_{F^N}$, we can assume
that $g=I$.

Suppose that the developing map of the $\OO{n}$-structure, in affine coordinates,
is $\left(t_1,t_2\right) : \C{2} \setminus 0 \to \OO{n}$.
Because $g=I$, $t_1$ must be $F$-invariant, i.e.
$t_1 : S_F \to \C{}$ is a nonconstant rational function.
Therefore $F$ must be a hyperresonant map, say
\[
F\left(z_1,z_2\right)=\left(\lambda_1 z_1, \lambda_2 z_2\right),
\]
for some complex numbers $\lambda_1, \lambda_2$ 
with $0 < \left|\lambda_1\right|\le \left|\lambda_2\right| < 1$,
with hyperresonance $\lambda_1^{m_1}=\lambda_2^{m_2}$.
Moreover, $t_1=P(u)/Q(u)$ for some polynomials $P$ and $Q$,
where $u=z_1^{m_1}/z_2^{m_2}$.
The function $t_2$ must then satisfy
\[
t_2\left(\lambda_1 z_1, \lambda_2 z_2 \right)
=
t_2\left(z_1, z_2\right)
+
p\left(t_1\left(z_1,z_2\right),1\right). 
\]

The developing map is a local
biholomorphism, so $t_1 : S_F \to \Proj{1}$ is a submersion to $\Proj{1}$.
We can factor $t_1$ into $t_1 = (P/Q) \circ u$,
and so $P/Q : \Proj{1} \to \Proj{1}$ must be a local
biholomorphism, and so a linear fractional transformation.
After replacing our developing map and holonomy
by $\left((g,0) \, \dev, (g,0) \, \hol \, (g,0)^{-1}\right)$,
using an element $(g,0) \in G$, we can arrange
that $P(u)/Q(u)=u$, i.e. $t_1\left(z_1,z_2\right)=u=z_1^{m_1}/z_2^{m_2}$.
Clearly $u : \C{2} \setminus 0 \to \Proj{1}$ must be a submersion,
and so $m_1=m_2=1$, i.e. the hyperresonance
is $\lambda_1=\lambda_2$, and we can
write $F\left(z_1,z_2\right)=\left(\lambda z_1, \lambda z_2\right)$
for some $\lambda \in \C{}$ with $0 < \left|\lambda\right| < 1$.

If we pick any point where $\left(t_1,t_2\right)$
are not defined as complex valued functions,
then at that point $\left(s_1,s_2\right)$ must be
defined. So
\[
\left(s_1,s_2\right)=
\left(\frac{z_2}{z_1},\frac{z_2^n t_2}{z_1^n}\right)
\]
must be defined. In particular, $z_2^n s_2$
must be defined at such a point. So the function $T=z_2^n s_2$ is 
defined and holomorphic at every point of $\C{2} \setminus 0$
and so at every point of $\C{2}$. Since we
know how $t_2$ behaves under holonomy action, we find
\[
T\left(\lambda z_1,\lambda z_2\right)
=
\lambda^n T\left(z_1, z_2\right)+\lambda^n p\left(z_1,z_2\right).
\]
Expanding $T$ into a power series, we find that $p=0$.

So now $(g,p)=(I,0)$, and
$t_1$ and $t_2$ are both $F$-invariant meromorphic functions on $\C{2} \setminus 0$, 
i.e. meromorphic functions on $S_F$, i.e. rational functions
of $u=z_1^{m_1}/z_2^{m_2}$, so $d t_1 \wedge dt_2=0$.
But then $\dev = \left(t_1,t_2\right)$ is not a local
biholomorphism.
%
%
%
%
%
%
%
%
\end{proof}

\begin{lemma}\label{lemma:ZeroOrMonic}
The holonomy generator $(g,p)$ of any $\OO{n}$-structure on
any diagonalizable Hopf surface, 
up to conjugation, has $p=0$ or $p$ a monic monomial.
\end{lemma}
\begin{proof}
By lemma~\vref{lemma:holonomyNormalForm},
either $p=0$ or $p$ is monomial or $g$ has finite order.
Finite order $g$ is impossible by lemma~\vref{lemma:InfiniteOrder}.
\end{proof}

\begin{lemma}
The holonomy of any $\OO{n}$-structure on any diagonalizable Hopf surface is generic.
\end{lemma}
\begin{proof}
We can assume that the map $F$ determining our Hopf surface is linear, diagonal,
\[
F\left(z_1,z_2\right)=\left(\lambda_1 z_1, \lambda_2 z_2\right),
\]
and hyperresonant, with hyperresonance $\left(m_1,m_2\right)$.
Suppose that $F : \C{2} \to \C{2}$ is 
a diagonal linear map in the Poincar\'e 
domain, $F(z)=\left(\lambda_1 z_1, \lambda_2 z_2\right)$. 
Suppose that 
$\dev : \C{2} \setminus 0 \to \OO{n}$ 
is the developing map of an 
$\OO{n}$-structure on the associated Hopf 
surface $S_F$, with holonomy generator
$\hol=(g,p)$. Then in affine coordinates, 
this developing map has the form 
$\left(t_1,t_2\right)$, where
$t_1$ must be a section of a flat projective
line bundle associated to $(F,g)$.
By proposition~\vref{proposition:MeromorphicSections},
\[
g=
\begin{pmatrix}
\lambda_1^{k_1} \lambda_2^{k_2} c & 0 \\
0 & c
\end{pmatrix}
\]
for some nonzero complex number $c$,
and
\[
t_1 = z_1^{k_1} z_2^{k_2} \frac{P_1(u)}{Q_1(u)}
\]
for $u=z_1^{m_1}/z_2^{m_2}$. (If $g$ is 
not hyperresonant, we take $P_1(u)$
and $Q_1(u)$ to be constants.)
We can assume that neither of $P_1$ and $Q_1$ 
have any double roots, or roots at $u=0$, and 
that neither of them have any common roots, and 
that $k_1=-1,0,1$ and $k_2=-1,0,1$ as before 
since $t_1 : \C{2} \backslash 0 \to \Proj{1}$ 
is a submersion.

By lemma~\vref{lemma:ZeroOrMonic}, if the holonomy $(g,p)$ is not generic, then
we arrange that $p$ is a monic monomial, say
\[
g=
\begin{pmatrix}
\lambda_1^{k_1} \lambda_2^{k_2} c & 0 \\
0 & c
\end{pmatrix}
\]
for some nonzero complex number $c$, and
\[
p\left(Z_1,Z_2\right)=Z_1^k Z_2^{n-k},
\]
and
\[
t_1 = z_1^{k_1} z_2^{k_2} \frac{P_1(u)}{Q_1(u)}
\]
for $u=z_1^{m_1}/z_2^{m_2}$. 
We can assume that neither of these polynomials 
have any double roots, or roots at $u=0$, and 
that neither of them have any common roots, and 
that $k_1=-1,0,1$ and $k_2=-1,0,1$. In order
that $p$ be resonant, we will need
\[
\left(\lambda_1^{k_1} \lambda_2^{k_2} c\right)^k c^{n-k}=1,
\]
i.e.
\[
c^n = \lambda_1^{-k k_1} \lambda_2^{-k k_2}.
\]

Next consider $t_2$. At this stage, we can see that
\[
t_2\left(\lambda_1 z_1, \lambda_2 z_2\right)
=
\lambda_1^{k k_1}
\lambda_2^{k k_2}
t_2\left(z_1,z_2\right)
+
z_1^{k k_1}
z_2^{k k_2}
\frac{P_1(u)^k}{Q_1(u)^k}.
\]
Let
\[
f\left(z_1,z_2\right)=
\frac{t_2\left(z_1,z_2\right)}{z_1^{k k_1}
z_2^{k k_2}
\frac{P_1(u)^k}{Q_1(u)^k}.
}
\]
Then compute out
\[
f(F(z))=f(z)+\frac{1}{\lambda_1^{k k_1} \lambda_2^{k k_2}},
\]
so that $f$ is a meromorphic section of 
a flat projective line bundle, 
\[
\left(\C{2} \setminus 0\right)_{(F,g')} \Proj{1},
\]
where
\[
g'=
\begin{bmatrix}
\lambda_1^{k k_1} \lambda_2^{k k_2} & 1 \\
0 & \lambda_1^{k k_1} \lambda_2^{k k_2}
\end{bmatrix}.
\]
By proposition~\vref{proposition:MeromorphicSections}, 
the only meromorphic section of this line bundle
is $f=\infty$, a contradiction.
\end{proof}

Summing up:

\begin{proposition}
The only $\OO{n}$-structures on diagonalizable Hopf surfaces are 
\begin{enumerate}
\item the radial structures, 
\item the eigenstructures and
\item the hyperresonant structures in table~\vref{table:exoticStructures}.
\end{enumerate}
\end{proposition}

\subsection{Exceptional Hopf surfaces}

\subsubsection{Diagonalizable holonomy on 
exceptional Hopf surfaces}

\begin{proposition}
Up to isomorphism, the only 
$\OO{n}$-structure on an exceptional 
Hopf surface which has 
diagonalizable holonomy 
is the eigenstructure.
\end{proposition}
\begin{proof}
Suppose that 
$\left(t_1,t_2\right) : \C{2} \backslash 0 \to \OO{n}$ 
is the developing map of an
$\OO{n}$-structure on an 
exceptional Hopf surface 
$S_F$, where 
$F\left(z_1,z_2\right)
=
\left(
  \lambda z_1, \lambda^m z_2 + z_1^m
\right)$. 
Suppose that the holonomy 
is $(g,p)$, and that $g$ is diagonalizable, say
\[
g =
\begin{pmatrix}
 a_1 & 0 \\
0 & a_2
\end{pmatrix}.
\]
By proposition~\vref{proposition:MeromorphicSections}, 
up to isomorphism we must have 
\[
 t_1 = c \, z_1^k,
\]
some integer $k$, and 
$a_1/a_2 = \lambda^k$. But then 
either $t_1$ is branched, 
along $z_1=0$, if $k>1$, or else
\[
 s_1 = \frac{1}{t_1} = \frac{1}{c \, z_1^k}
\]
is branched if $k<1$. Up to 
isomorphism, we can therefore 
ensure that $k=1$, and that
\[
 t_1 = z_1,
\]
and $a_1/a_2=\lambda$ and
\[
 t_2(F(z))=\frac{t_2(z)}{a_2^n} + p\left(\lambda z_1,1\right).
\]
By the usual trick of writing 
$t_2$ in terms of Weierstrass polynomials,
\[
 t_2(z) = h(z) \frac{W_1(z)}{W_2(z)}
\]
with $W_1(z)$ and $W_2(z)$ 
polynomial in the variable $z_2$, we see that
$W_2(z)$ must transform under 
composition with $F$ by scaling, so
\[
W_2(z) = c \, z_1^k
\]
for some constant $c \ne 0$ 
and integer $k \ge 0$. We can write
\[
 t_2(z) = \frac{T(z)}{z_1^k}
\]
for some holomorphic function 
$T(z)$ defined near the origin. Calculate that
\[
 T(F(z)) = 
\frac{\lambda^k}{a_2^n} T(z) 
+ 
\lambda^k 
z_1^k 
p\left(\lambda z_1,1\right).
\]
Differentiate both sides with
respect to $z_2$ to find 
\[
 \pd{T}{z_2}(F(z)) 
= \frac{\lambda^{k-m}}{a_2^n} \pd{T}{z_2}(z).
\]
so that $\pd{T}{z_2}$ is a section
of a line bundle over an exceptional
Hopf surface, so
\[
\pd{T}{z_2} = c \, z_1^{\ell}
\]
for some $\ell \ge 0$ and constant $c$ and
\[
a_2^n = \lambda^{k-\ell-m}.
\]
So
\[
 T = c \, z_1^{\ell} z_2 + T_1\left(z_1\right),
\]
for some holomorphic function 
$T_1\left(z_1\right)$, which is then forced to satisfy
\[
T_1\left(\lambda z_1 \right) 
= 
\lambda^{\ell+m} T_1\left(z_1\right)
- c \lambda^{\ell} z_1^{\ell+m}
+
\lambda^k z_1^k p\left(\lambda z_1, 1\right).
\]
Expand out $p$ as
\[
p\left(z_1,1\right)= \sum_{j=0}^{n} C_j z_1^j,
\]
and
\[
T_1\left(z_1\right)=
\sum_{j=0}^{\infty}
b_j z_1^j,
\]
(with the understanding that $C_j=0$ when 
$j<0$ or $j>n$ and that $b_j=0$ when $j<0$)
to see that
\[
\left( \lambda^j - \lambda^{\ell+m} \right) b_j
=
C_{j-k} \lambda^k - c \, \lambda^{\ell} \delta_{j=\ell+m}.
\]
If we plug in $j=\ell+m$, we find
\[
c = C_{\ell+m-k} \lambda^{k-\ell},
\]
and $b_{\ell+m}$ is arbitrary. 
For all other values of $j \ne \ell + m$,
\[
b_j = 
\frac{C_{j-k}}{\lambda^{j-k} - \lambda^{\ell+m-k}}.
\]

We now see that $a_1 = \lambda a_2$ 
and that $a_2^n=\lambda^{k - \ell - m}$, 
giving the eigenvalues
of $g$. Therefore $(g,p)$ is generic 
as long as either $k > \ell + m$ 
($g$ contracting) or
$n+k < \ell + m$ ($g$ expanding). 
So we can assume that $p=0$ or else that
$k \le \ell + m \le n+k$. If 
$a_1^i a_2^{n-i}=1$, then 
plugging in $a_1$ and then $a_2^n$, we find that
$i=\ell+m-k$, and 
$0 \le i \le n$. By lemma~\vref{lemma:holonomyNormalForm} there is only 
this one possible value of $i$ giving a coefficient $C_i$ of $p$ 
which we can't assume is $0$ without loss of generality. Therefore we can arrange
\[
p\left(z_1,1\right)=C_{\ell+m-k} z_1^{\ell+m-k}.
\]
and
\[
t_2=C_{\ell+m-k} \lambda^{k-\ell} \, z_1^{\ell-k} z_2 + b \, z_1^{\ell+m-k}.
\]
for some complex constant $b=b_{\ell+m}$.

To see if this structure is branched, compute
\[
 dt_1 \wedge dt_2 =
C_{\ell+m-k} \lambda^{k - \ell} z_1^{\ell-k} \, dz_1 \wedge dz_2.
\]
So the structure is unbranched except possibly at $z_1=0$. Since 
$s_1=\frac{1}{z_1}$, $s_1$ is not defined at $z_1=0$, and therefore 
we cannot use the coordinates $s_1, s_2$ to fix up the 
branch locus at $z_1=0$. Therefore $t_1$ and $t_2$ must be defined
at $z_1=0$, so that $k\le \ell$. Moreover $dt_1 \wedge dt_2$ can't
vanish at $z_1=0$, so $k=\ell$, yielding
\[
 t_2 = C_m z_2 + b \, z_1^m,
\]
and 
\[
 p\left(z_1,1\right)=C_m z_1^m.
\]
We can conjugate by a suitable isomorphism to arrange that $C_m=1$ and that $b=0$, so that 
our $\OO{n}$-structure is the eigenstructure on the exceptional Hopf surface.
\end{proof}

\subsubsection{Nondiagonalizable holonomy on exceptional Hopf surfaces}

\begin{proposition}
The only exceptional Hopf surfaces which admit $\OO{n}$-structures
with holonomy $(g,p)$ with $g$ nondiagonalizable are the linear 
nondiagonalizable Hopf surfaces. Up to isomorphism, the only 
such $\OO{n}$-structures they admit are the radial ones.
\end{proposition}
\begin{proof}
Suppose that $\left(t_1,t_2\right) : \C{2} \backslash 0 \to \OO{n}$ 
is the developing map of an $\OO{n}$-structure on an 
exceptional Hopf surface $S_F$, where 
$F\left(z_1,z_2\right)=\left(\lambda z_1, \lambda^m z_2 + z_1^m\right)$. 
Suppose that the holonomy is $(g,p)$, and that $g$ is not diagonalizable, say
\[
g=
\begin{pmatrix}
 a & 1 \\
0 & a
\end{pmatrix}.
\]
By lemma~\vref{lemma:notDiagonalHolonomy}, we can assume either (1) $(g,p)=(g,0)$ or else (2) 
\[
(g,p)=
\left(
\begin{pmatrix}
1 & 1 \\
0 & 1
\end{pmatrix},
Z_1^n
\right).
\]
In either case, $t_1$ is a meromorphic section of the obvious flat projective line
bundle. Let's consider case (2). By proposition~\vref{proposition:MeromorphicSections},
\[
t_1 = z_2 \left(\frac{\lambda}{z_1}\right)^m+c
\]
for some constant $c$. In particular, $t_1=\infty$ at $z_1=0$. Therefore at $z_1=0$, $s_1$ and $s_2$
must be holomorphic. Clearly
\[
s_1 = \frac{z_1^m}{\lambda^m z_2 + c \, z_1^m}.
\]
Therefore $s_1$ branches along $z_1=0$ unless $m=1$. We can check that 
\[
s_2\left(F\left(z_1,z_2\right)\right)
=
\frac{s_2\left(z_1,z_2\right)+\left(1+s_1\left(z_1,z_2\right)\right)^n}{\left(1+s_1\left(z_1,z_2\right)\right)^n}.
\]
In particular, along the line $z_1=0$, we find that $s_1=0$ and so
\[
s_2\left(0,\lambda z_2\right)
=
s_2\left(0,z_2\right)+1.
\]
Since $s_2$ is holomorphic on the entire line $z_1=0$, except perhaps at $z_2=0$, we can compute a Laurent
series expansion for $s_2$ and see that the constant term is inconsistent.

Therefore we can assume that we are in case (1): $p=0$. By proposition~\vref{proposition:MeromorphicSections}, 
up to isomorphism we must have 
\[
 t_1 = \frac{z_2}{a} \left(\frac{\lambda}{z_1}\right)^m.
\]
But then
\[
 s_1 
= 
\frac{1}{t_1} 
= 
\frac{a}{z_2} \left(\frac{z_1}{\lambda}\right)^m
\]
is branched unless $m=1$. So now let's assume that $m=1$. Our Hopf surface is linear but not diagonalizable, and 
\[
 t_1 = \frac{\lambda z_2}{a z_1}.
\]
and
\[
t_2\left(F(z) \right ) 
= 
\frac{t_2\left(z_1,z_2\right)}{a^n}.
\]
By 
proposition~\vref{proposition:MeromorphicSections}, 
this ensures that 
$t_2=c \, z_1^k$ for some integer $k$, and that
\[
 \frac{1}{a^n} = \lambda^k.
\]

On the line $z_1=0$, $t_1$ is infinite, so $s_1$ and $s_2$ must be finite. 
Similarly, on $z_2=0$, $s_1$ is infinite, so $t_1$ and $t_2$ must be finite, and 
have linearly independent differentials. Note that
\[
s_1 = 
\frac{1}{t_1} =
\frac{a z_1}{\lambda z_2}
\]
and
\[
s_2 = \frac{t_2}{t_1^n} =
\left(\frac{a z_1}{\lambda z_2}\right)^n c z_1^k.
\]
Along $z_1=0$, $s_2$ must be finite and $ds_1 \wedge ds_2 \ne 0$. In 
particular $t_2$ has a pole of order no more than $n$ along $z_1=0$. Compute
\[
 dt_1 \wedge dt_2 = 
-\frac{c \lambda k}{a} z_1^{k-2} dz_1 \wedge dz_2.
\]
The only possible zero of this holomorphic 2-form occurs along the line $z_1=0$, 
but $t_1$ and $t_2$ are not defined there, so we turn to $s_1$ and $s_2$ to see what 
happens near $z_1=0$. Compute
\[
ds_1 \wedge ds_2 
=
ck 
\frac%
{%
  a^{n+1} z_1^{k+n}
}%
{%
  \lambda^{n+1} z_2^{n+2}
}%
dz_1 \wedge dz_2.
\]
To get this to give a finite nonzero value along $z_1=0$, we need $k=-n$. 
Finally, composing with $\left(\lambda_0 \, I,0\right)$ where $\lambda_0^n=c$
gives a developing map which is identical to the developing map of the radial structure,
and gives the same holonomy.
\end{proof}

Summing up:
\begin{corollary}
Up to isomorphism, the only $\OO{n}$-structures on any exceptional Hopf surface are 
\begin{enumerate}
\item the eigenstructure and 
\item on a linear exceptional Hopf surface, the radial structure.
\end{enumerate}
\end{corollary}

This completes the proof of theorem~\vref{theorem:Classification}.

\section{Locally homogeneous geometric structures inducing these structures}\label{section:Inducement}

\begin{lemma}
 The Zariski closure of the subgroup of $\GL{2,\C{}}$ generated by a matrix
\[
 g = \begin{pmatrix}
     a_1 & 0 \\
     0 & a_2
     \end{pmatrix}
\]
with neither $a_1$ nor $a_2$ on the unit circle is 
\[\begin{array}{ll}
     \SetSuchThat{\begin{pmatrix}
             Z_1 & 0 \\
             0 & Z_2
            \end{pmatrix}}{Z_1^{n_1}=Z_2^{n_2}} & 
\text{ if $g$ has hyperresonance } a_1^{n_1}=a_2^{n_2}\\
     \text{the diagonal matrices} & \text{ if $g$ is not hyperresonant.}
\end{array}
\]
\end{lemma}
\begin{proof}
Suppose that $p\left(Z_1,Z_2\right)$ is a complex polynomial vanishing on all of the points 
$\left(Z_1,Z_2\right)=\left(a_1^k,a_2^k\right)$ for all integers $k$. From among all monomials
$Z_1^{j_1} Z_2^{j_2}$ which occur in $p$ with nonzero coefficient, pick one for which $a_1^{j_1} a_2^{j_2}$ 
is largest in absolute value. Let
\[
f_n\left(Z_1,Z_2\right)=
\frac{1}{n} \sum_{k=0}^{n-1} 
\frac
{%
  p\left(Z_1^k, Z_2^k\right)
}%
{%
  Z_1^{k \, j_1} Z_2^{k \, j_2}
}.
\]
Then $f_n\left(a_1^k, a_2^k\right)=0$ for all integers $k$. Consider how each monomial in $p$ contributes to 
$f_n$. A monomial $Z_1^{\ell_1} Z_2^{\ell_2}$ yields a term
\[
 \frac{1}{n} 
\sum_{k=0}^{n-1} 
  Z_1^{k \, \left(\ell_1-j_1\right)} 
  Z_2^{k \left(\ell_2 - j_2\right)}.
\]
Write $a_j=r_j e^{i \theta_j}$. At 
$\left(Z_1,Z_2\right)
=
\left(a_1, a_2\right)$, 
this term yields
\[
\frac{1}{n} 
\sum_{k=0}^{n-1} 
  r_1^{k \, \left(\ell_1-j_1\right)} 
  r_2^{k \left(\ell_2 - j_2\right)}
  e^{i k \left(\ell_1-j_1 + \ell_2 - j_2 \right)}.
\]
This term goes to $0$ as 
$n \to \infty$ unless 
$r_1^{\ell_1 - j_1} r_2^{\ell_2-j_2}=1$ 
and 
$\ell_1-j_1 + \ell_2-j_2$ 
is a multiple of $2 \pi$, i.e. vanishes. In particular, the term 
coming from the monomial $Z_1^{j_1} Z_2^{j_2}$ yields a nonzero contribution 
in the limit. But in the limit $f_n\left(a_1,a_2\right) \to 0$, so some other monomial must 
cancel $Z_1^{j_1} Z_2^{j_2}$. Therefore there must be some pairs $\left(j_1,j_2\right)$ and 
$\left(\ell_1,\ell_2\right)$ for which $a_1^{j_1} a_2^{j_2} = a_1^{\ell_1} a_2^{\ell_2}$. 
So $g$ is hyperresonant. The terms in $f_n$ which don't vanish in the limit as 
$n \to \infty$ must all have powers of $Z_1$ and $Z_2$ differing from 
$\left(j_1, j_2\right)$ by integer multiples of the hyperresonance of $g$.

We can grade each monomial $Z_1^{k_1} Z_2^{k_2}$, by sliding $\left(k_1,k_2\right)$ over by 
integer multiples of the hyperresonance until we make $k_1$ as small as possible, 
and using the resulting $k_1$ value as the grading. We can write each polynomial 
$p\left(Z_1,Z_2\right)$ as a sum of graded pieces. Suppose 
$p\left(Z_1,Z_2\right)$ vanishes on all of the points $\left(a_1^k,a_2^k\right)$. 
Let's write $p\left(Z_1,Z_2\right)=\sum p_j\left(Z_1,Z_2\right)$ as a sum of graded pieces. 
Consider again these functions $f_n\left(Z_1,Z_2\right)$. Taking the limit 
\[
0 = \lim_{n \to \infty} f_n\left(a_1,a_2\right)
\]
only the terms from the highest graded piece enter into the limit. If
\[
 p_N\left(Z_1,Z_2\right)
=
\sum_{\ell_1, \ell_2} 
  c_{\ell_1 \ell_2} 
  Z_1^{\ell_1} 
  Z_2^{\ell_2}
\]
is the highest graded piece, then
\[
0 
= 
\lim_{n \to \infty} 
  f_n\left(a_1,a_2\right)
=
\sum 
  c_{\ell_1 \ell_2}.
\]
Modulo the hyperresonance relation $Z_1^{n_1}-Z_2^{n_2}$, each term in $p_N\left(Z_1,Z_2\right)$ can be 
shifted over to become a multiple of one single term:
\[
 p_N\left(Z_1,Z_2\right)
=
Z_1^{N_1} 
Z_2^{N_2} 
\sum_{\ell_1, \ell_2} 
  c_{\ell_1 \ell2} 
= 0.
\]
\end{proof}

Let $H_0 \subset \GL{2,\C{}}$ be the subgroup
fixing a point of $\C{2} \setminus 0$.
There is an obvious Lie group morphism $g \in \GL{2,\C{}} \mapsto (g,0) \in G$
where as above
\[
 G=\Gn.
\]

\begin{lemma}
Every $\OO{n}$-structure on any linear Hopf surface is induced by its $\GL{2,\C{}}/H_0$-structure.
\end{lemma}
\begin{proof}
We map $\Phi : \C{2} \setminus 0 \to \OO{n}$ by the identity map in affine coordinates, and then map $\Phi : \GL{2,\C{}} \to G$ by the embedding
$g \mapsto (g,0)$.
\end{proof}

\begin{lemma}
The holonomy group of any $\OO{n}$-structure on any Hopf surface $S_F$ is  
contained in $\SL{2,\C{}}/\text{$n$-th roots of 1}$, precisely if the $\OO{n}$-structures is
the eigenstructure of 
\begin{enumerate}
 \item a hyperresonant linear map $F$ with eigenvalues $\lambda_1$ and $\lambda_2$ and hyperresonance either 
\begin{enumerate}
\item 
$\lambda_1^{n}=\lambda_2^2$ or
\item
$\lambda_1^{n/2}=\lambda_2$,
\end{enumerate}
or
\item a nondiagonalizable linear map $F$ when $n=2$, i.e. an $\OO{2}$-structure.
\end{enumerate}
\end{lemma}
\begin{proof}
Just take determinants of the holonomy generators. Note that there are two distinct eigenstructures
for a diagonalizable linear map $F$, corresponding to the two distinct eigenspaces.
\end{proof}

\begin{proposition}
Suppose that some $G'/H'$-structure induces the $\GL{2,\C{}}/H_0$-structure
on some linear Hopf surface. Then $G' \to \GL{2,\C{}}$ is onto or else
$G' \to \SL{2,\C{}}$ is onto.
\end{proposition}
\begin{proof}
There are no other subgroups of $\GL{2,\C{}}$ which act transitively on $\C{2} \backslash 0$;
see Huckleberry and Livorno \cite{Huckleberry/Livorno:1981}.
\end{proof}

\begin{remark}
The eigenstructure on a linear Hopf surface $S_F$ is induced, as we have already proven,
by the $G_0/H_0$-structure, where $G_0$ is the group of linear transformations
of $\C{2}$ preserving an eigenspace of the linear map $F$. 
\end{remark}

\begin{lemma}
The radial structure on the generic Hopf surface is induced by the 
$\GL{2,\C{}}/H_0$-structure on $\C{2} \backslash 0$ and by no proper subgroup
of $\GL{2,\C{}}$.
\end{lemma}
\begin{proof}
For a generic Hopf surface, the radial and eigen structures
will have holonomy generator Zariski dense in the diagonal matrices.
That ensures that for any holomorphic reduction, say to a $G'/H'$-structure, 
$G'$ will have to map onto a subgroup of $G$ containing the diagonal matrices. 
Moreover, $G'$ will have an open orbit in $\OO{n}$, containing at least
the open orbit of the diagonal matrices, which is everything except the fibers
over $0$ and $\infty$ and the $0$-section. However, the radial structure has everything
but the $0$-section in its image, so we will need $G'$ to have as image a larger group than
just the diagonal subgroup. Indeed our group $G'$ will need to act
transitively on $\Proj{1}$, so must map onto $\PSL{2,\C{}}$
by the classification of homogeneous surfaces (see Huckleberry and
Livorno \cite{Huckleberry/Livorno:1981}). Any subgroup 
of $\GL{2,\C{}}/\text{$n$-th roots of 1}$ mapping onto $\PSL{2,\C{}}$
will have to contain $\SL{2,\C{}}/\text{$n$-th roots of 1}$ and therefore the image of $G'$ must 
contain all of $\GL{2,\C{}}/\text{$n$-th roots of 1}$. 
\end{proof}

\section{Conclusions}

We have found all of the $\OO{n}$-structures on all Hopf surfaces explicitly, by computed
their developing maps and holonomy groups explicitly. This makes it possible to determine which of these
structures are induced from other locally homogeneous geometric structures on Hopf
surfaces. The one surprising result of the classification is the appearance of the
hyperresonant $\OO{n}$-structures (on the hyperresonant Hopf surfaces). The
hyperresonant $\OO{n}$-structures have no apparent geometric or intuitive description.
They depend on the prescence of complicated meromorphic functions (rational functions in
the canonical affine structure), and so disappear on the Hopf surfaces with trivial function fields.

The relation of these results to Wall's results \cite{Wall:1985,Wall:1986} deserves some clarification.
The full picture, of all holomorphic locally homogeneous geometric structures on
compact complex surfaces, and which are induced from which, is still hidden. It seems
likely that this picture will soon become clear. 
The classification of holomorphic Cartan geometries on compact complex surfaces 
would then appear to be within reach. We have to keep in mind that the explicit
classification of holonomy morphisms and developing maps for holomorphic projective 
connections on complex algebraic curves is still unknown, and perhaps too
complicated to be classifiable (see \cite{Gallo/Kapovich/Marden:2000}). Therefore we would only hope to classify
holomorphic Cartan geometries on compact complex surfaces modulo the classification on curves.
It seems likely that holomorphic Cartan geometries can be classified on linear 
Hopf manifolds in all dimensions.

\nocite{*}
\bibliographystyle{amsplain}
\bibliography{hopf}

\end{document}